\RequirePackage{fix-cm}
\documentclass[smallextended,numbook,envcountsame]{svjour3}    
\smartqed  
\usepackage{graphicx}
 \usepackage{mathptmx}      
 \journalname{Probability Theory and Related Fields}
\usepackage{amsmath,amssymb,amsfonts,mathrsfs}
\RequirePackage[numbers,comma,square,sort&compress]{natbib}

\usepackage{graphicx,tikz,subfigure,color}
\usetikzlibrary{decorations.pathreplacing}

\usepackage{nicefrac}
\usepackage{enumerate}

\RequirePackage{color}
\RequirePackage[colorlinks, urlcolor=my-blue,linkcolor=my-link,citecolor=my-green]{hyperref}
\definecolor{my-link}{rgb}{0.5,0.0,0.0}
\definecolor{my-blue}{rgb}{0.0,0.0,0.6}
\definecolor{my-red}{rgb}{0.5,0.0,0.0}
\definecolor{my-green}{rgb}{0.0,0.5,0.0}
\definecolor{nicos-red}{rgb}{0.75,0.0,0.0}
\definecolor{light-gray}{gray}{0.6}
\definecolor{really-light-gray}{gray}{0.8}

\newcommand{\be}{\begin{equation}}
\newcommand{\ee}{\end{equation}}

\providecommand{\abs}[1]{\vert#1\vert}

\providecommand{\pp}[1]{\langle#1\rangle}
\newcommand{\fl}[1]{\lfloor{#1}\rfloor}

\def\cG{\mathcal{G}}

  \def\cT{\mathcal{T}}

\def\Phat{\widehat\P}
\def\Omhat{\widehat\Omega}

\def\what{\widehat\w}

\def\t{\mathfrak t}
\def\tmin{{\underline\t}}
\def\tmax{{\bar\t}}

\def\kS{\mathfrak{S}}

\def\bE{\mathbb{E}}
\def\bN{\mathbb{N}}
\def\bP{\mathbb{P}}

\def\bR{\mathbb{R}}
\def\bZ{\mathbb{Z}}

 \def\Z{\bZ}  \def\R{\bR}\def\N{\bN}

\def\w{\omega}

\def\e{\varepsilon}

\def\ri{{\mathrm{ri\,}}}

\def\m1{\mathbf{1}}

\def\kS{\mathfrak S}

 \def\Vvv{{\rm\mathbb{V}ar}}

\def\E{\bE}
\def\P{\bP} 

\def\funct lp{L} 
\def\funct lpbar{\bar L} 

\def\range{\mathcal R}
\def\Uset{\mathcal U}
\def\Diff{\mathcal D}
\def\EP{\mathcal E}

\def\cV{\mathcal V}

\def\gpp{{g_{\text{pp}}}}

\def\Gpp{G}
\def\gpp{g_{\text{\rm pp}}}

\def\ri{{\mathrm{ri\,}}}

\def\ximin{{\underline\xi}}
\def\ximax{{\overline\xi}}
\def\zetamin{{\underline\zeta}}
\def\zetamax{{\overline\zeta}}

\def\B{{B}}

\def\Ombig{{\widehat\Omega}}
\def\Pbig{{\widehat\P}}
\def\Ebig{{\widehat\E}}
\def\kSbig{{\widehat\kS}}
 
\def\OBPbig{(\Ombig, \kSbig, \Pbig)}

\def\cid{\xi_*}
\def\Bci{B^*}

\def\cidmin{\underline\xi_{^{\scriptstyle*}}}
\def\cidmax{\overline\xi_*}

\def\pp{p}  
\def\Ew{m_0}   
\def\opc{\vec p_c} 

\def\etamin{{\underline\eta}}
\def\etamax{{\overline\eta}}

\def\cidl{\cid^{\text\rm{(l)}}}
\def\cidr{\cid^{\text\rm{(r)}}}

\def\cT{\mathcal T}

\definecolor{darkgreen}{rgb}{0.0,0.5,0.0}
\definecolor{darkblue}{rgb}{0.0,0.0,0.3}
\definecolor{nicosred}{rgb}{0.65,0.1,0.1}
\definecolor{light-gray}{gray}{0.7}
\allowdisplaybreaks[1]

\allowdisplaybreaks[1]
\newenvironment{proofof}[2]{\removelastskip\vspace{6pt}\noindent
 {\it Proof  #1.}~\rm#2}{\par\vspace{6pt}}

\begin{document}

\title{Geodesics and the competition interface for the corner growth model
}

\titlerunning{Geodesics}        

\author{N.\ Georgiou         \and
        F.\ Rassoul-Agha \and
        T.\ Sepp\"al\"ainen 
}

\authorrunning{N.\ Georgiou, F.\ Rassoul-Agha, and T.\ Sepp\"al\"ainen} 

\institute{N.\ Georgiou \at
              School of Mathematical and Physical Sciences, Department of Mathematics, University of Sussex, Falmer Campus, Brighton BN1 9QH, UK \\
              \email{n.georgiou@sussex.ac.uk}           
           \and
           F.\ Rassoul-Agha \at
              Mathematics Department, University of Utah, 155 South 1400 East, Salt Lake City, UT 84109, USA \\
              \email{firas@math.utah.edu}
           \and
           T.\ Sepp\"al\"ainen \at
              Mathematics Department, University of Wisconsin-Madison, Van Vleck Hall, 480 Lincoln Dr., Madison, WI 53706, USA \\
              \email{seppalai@math.wisc.edu}
}

\date{Received: August 2015 / Accepted: May 2016}  

\maketitle

\begin{abstract}
 We study  the directed last-passage percolation model on the planar square  lattice  with nearest-neighbor steps and general i.i.d.\ weights on the vertices,  outside of  the class of exactly solvable models.   Stationary cocycles are constructed for this percolation model from queueing fixed points.  These cocycles serve as boundary conditions for stationary last-passage percolation,  solve  variational formulas that characterize limit shapes,  and   yield existence of  Busemann functions in directions where the shape has some regularity. 
In a sequel to this paper the cocycles are used to prove results about semi-infinite geodesics and the competition interface.

\end{abstract}


\section{Introduction}

 In the  {\sl corner growth model}, or    directed  nearest-neighbor  last-passage percolation (LPP)   on the lattice $\Z^2$,   i.i.d.\ random weights $\{\w_x\}_{x\in\Z^2}$ are  used to define {\sl last-passage times}  $\Gpp_{x,y}$ between  lattice points $x\le y$ in $\Z^2$ by 
 \be \Gpp_{x,y}=\max_{x_\centerdot}\sum_{k=0}^{n-1}\w_{x_k} . 
	\label{Gxy1}\ee
The maximum is over paths $x_\centerdot=\{x=x_0, x_1, \dotsc, x_n=y\}$ that satisfy  $x_{k+1}-x_k\in\{e_1,e_2\}$ (up-right paths). 
{\sl Geodesics} are paths that maximize in \eqref{Gxy1}.
Geodesics are unique  if $\w_x$ has a continuous distribution.   For $x\in\Z_+^2$,   the geodesic from $0$ to $x$ must go through either  $e_1$ or $e_2$.   These two clusters are separated by  the  {\sl competition interface} .    The purpose of this paper is to study   the geodesics and  competition interface for the case where the weights are {\sl general}, subject to a lower bound $\w_0\ge c$ and a moment condition $\E\abs{\w_0}^{2+\e}<\infty$.  
We address the key questions of existence, uniqueness, and coalescence of directional semi-infinite geodesics, nonexistence of doubly infinite geodesics, and the asymptotic direction of the competition interface.

Systematic study of geodesics in percolation began with the work of Licea and Newman \cite{Lic-New-96}.  
Their seminal work on  undirected first-passage percolation, summarized in Newman's ICM paper \cite{New-95},   
 utilized a global curvature assumption on the limit shape 
 to derive properties of geodesics, and as a consequence   the existence of  Busemann functions, which are  limits of gradients of passage times.     
Assuming $\w_0$ has a continuous distribution, 
they proved the existence of  a deterministic, full-Lebesgue-measure set of directions
for  which there is a unique geodesic out of every lattice point and that  geodesics in a  given direction from this set coalesce. Furthermore,
for any two such directions $\eta$ and $\zeta$ there are no doubly infinite geodesics whose two ends have directions $\eta$ and $-\zeta$.

The global curvature assumption cannot as yet be checked in percolation models with general weights,  but it can be  verified  in several models with special  features.
One such case is   Euclidean first passage percolation  based on a homogeneous Poisson point process. For this model, Howard and Newman \cite{How-New-01} 
showed that every geodesic has a direction and that in every  fixed direction there is at least one geodesic out of every lattice point.

A number of investigators have built on the  approach opened up by Newman et al.  This has led to impressive  progress in understanding geodesics,  Busemann functions, coalescence, competition, and stationary processes   in directed last-passage percolation models with enough explicit  features to enable verification of the curvature assumptions.  This work is on  models built on  Poisson point processes \cite{Wut-02,Cat-Pim-11,Cat-Pim-12, Bak-Cat-Kha-14,Pim-07}
and on the corner growth model with exponential weights  \cite{Fer-Pim-05,Fer-Mar-Pim-09,Cat-Pim-12,Cat-Pim-13,Pim-13-}.      In the case of the exponential corner growth model,  another set of tools comes from its connection with an  exactly solvable  interacting particle system, namely the  totally asymmetric simple exclusion process  (TASEP).  

 
 The competition interface of the exponential corner growth model maps to a second-class particle in TASEP, so this object has been studied from both perspectives.  
An early result of Ferrari and Kipnis  \cite{Fer-Kip-95} proved that the scaled location of the second-class particle in a rarefaction fan  converges in distribution to a uniform random variable. 
\cite{Mou-Gui-05} improved this to almost sure convergence with the help of  concentration inequalities and the TASEP variational formula from \cite{Sep-99-aop}.  
\cite{Fer-Pim-05}   gave a different proof of almost sure convergence by   applying the techniques of directed geodesics and then  obtained the distribution of the asymptotic direction of the competition interface from  the TASEP results of \cite{Fer-Kip-95}.  

Subsequently these results on the almost sure convergence of the competition interface and its limiting random angle were extended from the quadrant  to larger classes of initial profiles in two rounds: first by 
\cite{Fer-Mar-Pim-09} still with TASEP and geodesic techniques,  and then by 
\cite{Cat-Pim-13}  by applying their earlier  results  on Busemann functions \cite{Cat-Pim-12}.  
%
Coupier \cite{Cou-11}  also relied on the TASEP connection   to sharpen the geodesics  results of \cite{Fer-Pim-05}.  
He showed that there are no triple geodesics (out of the origin) in any direction and that every fixed direction has a unique geodesic.  

To summarize, the common thread  of the work above   is the  use of explicit   curvature of the limit shape  to control directional geodesics.  Coalescence of geodesics  leads to   Busemann functions  and   stationary versions of the percolation process.   In exactly solvable cases, such as  the exponential corner growth model,   information  about the distribution of the Busemann functions  is powerful.   For example, it enables calculation of    the distribution of the asymptotic direction of the competition interface \cite{Fer-Pim-05,Fer-Mar-Pim-09,Cat-Pim-13} and  to get bounds on the coalescence time of geodesics \cite{Pim-13-,Wut-02}.  

 An   independent   line of work is that of  Hoffman \cite{Hof-05,Hof-08} on  undirected first passage percolation,  with  general  weights and without 
  regularity assumptions on the limit shape.   \cite{Hof-05} proved that there are at least two semi-infinite geodesics by deriving a contradiction from the assumption that all semi-infinite geodesics coalesce.   The technical proof involved the construction of a Busemann function.  (\cite{Gar-Mar-05} gave  an independent proof with a different method.)   \cite{Hof-08} extended this  to  at least four geodesics.
No further information about geodesics was obtained.   
  In another direction, \cite{Weh-97} restricted the number of doubly infinite geodesics to zero or infinity.


 The idea of studying geodesic-like objects  to produce stationary processes has also appeared in random dynamical systems.
 Article  \cite{E-etal-00} and its extensions \cite{Hoa-Kha-03,Itu-Kha-03,Bak-07, Bak-Kha-10,Bak-13} prove existence and uniqueness of semi-infinite minimizers of an action functional 
to conclude existence and uniqueness of an invariant measure for the Burgers equation with random forcing. 
These articles treat situations where the space is compact or essentially compact.  To make progress in the non-compact case,  the approach of Newman et al.\ was adopted again in \cite{Bak-Cat-Kha-14,Bak-15-}, as mentioned above.

A new approach to the problem of geodesics came in the work of 
Damron and Hanson \cite{Dam-Han-14} who constructed (generalized) Busemann functions from   weak subsequential limits of   first-passage time differences.  This gave  access to  properties of geodesics,   while 
weakening  the need for the global curvature assumption.  For instance, assuming differentiability and strict convexity of the limit shape, \cite{Dam-Han-14} proves that, with probability one, every semi-infinite geodesic has a direction and  for any given direction there exists a semi-infinite 
directed geodesic  out of every lattice point.   They   construct  a   tree  of semi-infinite geodesics in any given  direction such that from every lattice point   emanates a unique
geodesic in this tree  and the tree has no doubly infinite geodesics. 
However, since the Busemann functions of \cite{Dam-Han-14} are constructed from weak subsequential  limits,  
no claims about uniqueness  of directional geodesics  are made.  The geodesics constructed in their trees all coalesce, but one cannot infer from this that all geodesics in a given direction coalesce.

When  first-passage percolation  is restricted to the upper half plane, \cite{Weh-Woo-98} was the first to rule out the existence of doubly infinite geodesics.  
\cite{Auf-Dam-Han-15}  extended this half-plane  result to more general weight distributions and then applied it to prove coalescence in a tree of geodesics   constructed through a  limit, as in  \cite{Dam-Han-14}  discussed above.  The constructed tree of geodesics again has no infinite backward paths, but it is open to show that the geodesics are asymptotically directed in direction $e_1$.


 \smallskip 
 
 The approach of our work  is the opposite of   the approach  that relies on global curvature, and closer in spirit to \cite{Dam-Han-14}.    We begin by constructing the stationary versions of the percolation process in the form of   {\sl stationary cocycles}.   This comes from related results in  queueing theory \cite{Mai-Pra-03,Pra-03}.      Local regularity assumptions on the limit shape then give enough control to prove that these cocycles are also almost sure Busemann functions.  
 This was done in \cite{Geo-Ras-Sep-15a-}.  
 
  In the present  paper we continue the project by utilizing the cocycles and the Busemann functions to study  geodesics and the competition interface of  the corner growth model  with general weights.   In other words, what is achieved here is a  generalization of the results of  \cite{Cou-11,Fer-Pim-05} without the explicit solvability framework. 
  
   A key technical point   is that a family of {\sl cocycle geodesics}  can be defined locally    by following   minimal gradients of a cocycle.    The coalescence proof of \cite{Lic-New-96} applies to cocycle geodesics.    Monotonicity and continuity properties of these cocycle geodesics allow us to use them to control all geodesics.  In the end we reproduce many of the basic properties of geodesics, some with no assumptions at all and others with local regularity assumptions on the limit shape.  
   Note that, in contrast with the results for the explicitly solvable exponential case,    our results must take into consideration the possibility of corners and linear segments in the limit shape.   
 
To control  the competition interface we characterize it in terms of the cocycles,  as was done in  terms of Busemann functions in  \cite{Fer-Pim-05,Cat-Pim-13,Pim-07}.   Here again we can get interesting results even without regularity assumptions.  For example,   assuming that the weight $\w_x$ has continuous distribution,  the atoms of the asymptotic direction of the competition interface are exactly the corners of the limit shape.  Since the shape is expected to be differentiable, the conjecture is that the asymptotic direction has continuous distribution.  

To extend our results to ergodic  weights and higher dimensions, a  possible strategy that avoids the reliance on queueing theory would be to develop 
 sufficient   control on the gradients 
$\Gpp_{x,\fl{n\xi}}-\Gpp_{y,\fl{n\xi}}$ (or their point-to-line counterparts)  to construct cocycles through   weak limits as $n\to\infty$.  This  worked well for undirected first-passage percolation in   \cite{Dam-Han-14} because  the gradients are uniformly integrable.   Note however that when $\{\w_x\}$ are only   ergodic, the limiting 
shape can  have corners and linear segments, and can even be a finite polygon.

 \smallskip

{\bf Organization of the paper.}  
Section \ref{sec:results} describes the corner growth model and the main results of the paper.  
Section \ref{sec:cocycles}  states the existence and properties of the cocycles and Busemann functions on which all the results of the paper are based.   
Section \ref{sec:geod} studies cocycle geodesics and  proves our results for geodesics.  
Section \ref{sec:ci-pf}  proves results for the competition interface.  
 Section \ref{sec:solv}  derives the distributions of the asymptotic speed of the left and right competition interfaces for the corner growth model with geometric weights.  This is an exactly solvable case, but this particular feature has not been calculated in the past.
Appendix \ref{app} has auxiliary results such as an ergodic theorem for cocycles proved in \cite{Geo-etal-15-}.

\smallskip

{\bf Notation and conventions.}   $\R_+=[0,\infty)$,  $\Z_+=\{0,1,2,3, \dotsc\}$, $\N=\{1,2,3,\dotsc\}$. The standard basis vectors of $\R^2$ are  $e_1=(1,0)$ and $e_2=(0,1)$ and  the $\ell^1$-norm of $x\in\R^2$  is   $\abs{x}_1=\abs{x\cdot e_1} + \abs{x\cdot e_2}$.   For $u,v\in\R^2$ a  closed line segment on $\R^2$ is denoted by   $[u,v]=\{tu+(1-t)v:  t\in[0,1]\}$, and an open line segment by  $]u,v[=\{tu+(1-t)v:  t\in(0,1)\}$.    Coordinatewise ordering $x\le y$ means that $x\cdot e_i\le y\cdot e_i$ for both $i=1$ and $2$.  Its negation $x\not\le y$ means that   $x\cdot e_1> y\cdot e_1$ or   $x\cdot e_2> y\cdot e_2$.    An admissible or up-right finite path $x_{0,n}=(x_k)_{k=0}^n$, infinite path  $x_{0,\infty}=(x_k)_{0\le k<\infty}$, or doubly infinite path $x_{-\infty,\infty}=(x_k)_{k\in\Z}$  on $\Z^2$ satisfies $x_k-x_{k-1}\in\{e_1,e_2\}$ $\forall k$.  

The basic environment space is $\Omega=\R^{\Z^2}$ whose elements are denoted by $\w$.  There is also a larger product space $\Ombig=\Omega\times\Omega'$ whose elements are denoted by $\what=(\w, \w')$.   

 
A statement that  contains $\pm$ or $\mp$ is a combination of two statements: one for the top choice of the sign and another one for the bottom choice.  




\section{Main results}  
\label{sec:results}

\subsection{Assumptions} 
The two-dimensional  corner growth model is the last-passage percolation model on the planar square lattice $\Z^2$ with admissible steps   $\{e_1,e_2\}$.  
    $\Omega=\R^{\Z^2}$ is the   space of   environments  or weight configurations $\w=(\w_x)_{x\in\Z^2}$.   The group of spatial translations   $\{T_x\}_{x\in\Z^2}$  acts on $\Omega$ by $(T_x\w)_y=\w_{x+y}$ for $x,y\in\Z^2$.   Let $\kS$ denote the Borel $\sigma$-algebra of $\Omega$.  $\P$ is  a Borel probability measure    on $\Omega$  under which the weights  $\{\w_x\}$  are  independent, identically distributed (i.i.d.) nondegenerate random variables with a $2+\e$ moment.      Expectation under $\P$ is denoted by $\E$.   For a technical reason we also assume $\P(\w_0\ge c)=1$ for some finite constant $c$.   Here is the standing assumption,  valid throughout the paper: 
	\be\begin{aligned} \label{2d-ass}
		&\text{$\P$ is i.i.d., \, $\E[|\w_0|^{\pp}]<\infty$ for some $\pp>2$,\,     $\sigma^2=\Vvv(\w_0)>0$, and }\\
		 &\text{$\P(\w_0\ge c)=1$ for some   $c>-\infty$.} 
	\end{aligned} 
	\ee  
The symbol   $\w$ is reserved for   these $\P$-distributed  i.i.d.\ weights, also later when they are embedded in a larger configuration $\what=(\w, \w')$.  

The only reason for 
  assumption $\P(\w_0\ge c)=1$ is that Theorem \ref{th:cocycles} below is proved in 
   \cite{Geo-Ras-Sep-15a-} by applying  queueing theory.  
  In that context $\w_x$ is a service time and the results   have been proved only for $\w_x\ge 0$.  (The extension to $\w_x\ge c$  is immediate.)   
Once the queueing results are extended to general real-valued  i.i.d.\ weights $\w_x$ subject to the moment assumption in \eqref{2d-ass},  everything in this paper is true  for these general real-valued weights.  


  
  \subsection{Last-passage percolation} 
Given an environment $\w$ and two points $x,y\in\Z^2$ with $x\le y$ coordinatewise,
define the {\sl point-to-point last-passage time} by
	\[ \Gpp_{x,y}=\max_{x_{0,n}}\sum_{k=0}^{n-1}\w_{x_k}.\]
The maximum is over paths $x_{0,n}=(x_k)_{k=0}^n$  that start at  $x_0=x$,  end at $x_n=y$ with $n=\abs{y-x}_1$,  and have increments $x_{k+1}-x_k\in\{e_1,e_2\}$. We call such paths {\sl admissible} or {\sl up-right}.

According to the basic shape theorem  (Theorem 5.1(i) of \cite{Mar-04})   there exists a nonrandom continuous  function  $\gpp:\R_+^2\to\R$ such that  
\begin{align}\label{lln5}
\lim_{n\to\infty} n^{-1}\max_{x\in\Z_2^+:\,\abs{x}_1=n}\abs{\Gpp_{0,x}-\gpp(x)}=0\quad\text{$\P$-almost surely}.
\end{align}
The {\sl shape function}   $\gpp$ is  symmetric, concave, and  $1$-homogeneous (i.e.\ $\gpp(c\xi)=c\gpp(\xi)$ for $\xi\in\R_+^2$ and $c\in\R_+$).

\subsection{Gradients and convexity} 
Since $\gpp$ is homogeneous,  it is completely determined by its values on 
  $\Uset=\{te_1+(1-t)e_2:t\in[0,1]\}$,   the convex hull of $\range=\{e_1, e_2\}$.    The relative interior $\ri\Uset$ is the open line segment $\{te_1+(1-t)e_2:t\in(0,1)\}$.  
Let
\[  \Diff=\{\xi\in\ri\Uset:\text{ }\gpp\text{ is differentiable at $\xi$}\}\]
be the set of points at which  the gradient $\nabla\gpp(\xi)$
 exists in the usual sense of differentiability of functions of several variables.  
 By concavity the set $(\ri\Uset)\smallsetminus\Diff$ is at most countable.
 
 A point $\xi\in\ri\Uset$ is an {\sl exposed point} if there exists a vector $v\in\R^2$ such that
 	\be \label{eq:epod}	\begin{aligned}
	   \forall \zeta\in(\ri\Uset)\smallsetminus \{\xi\}: \; \gpp(\zeta)\; <\; \gpp(\xi) + v\cdot(\zeta - \xi) .   \end{aligned}\ee
The set  of {\sl exposed points of differentiability} of $\gpp$ is  
$\EP=  \{ \xi\in\Diff:  \text{\eqref{eq:epod} holds}\}$.  For $\xi\in\EP$ we have $v=\nabla\gpp(\xi)$ uniquely.
Condition \eqref{eq:epod}  is formulated   in terms of $\Uset$ because as a homogeneous function  $\gpp$   cannot have exposed points on $\R_+^2$. 	
		
It is expected but currently unknown that $\gpp$ is differentiable on all of $\ri\Uset$.  But  left and right gradients exist.    A {\sl left limit} $\xi\to\zeta$ on $\Uset$ means that 
$\xi\cdot e_1$ increases to $\zeta\cdot e_1$, while in a {\sl right limit}   
$\xi\cdot e_1$ decreases to $\zeta\cdot e_1$.  
  
For $\zeta\in\ri\Uset$ define one-sided gradient vectors  by
\begin{align*}
\nabla\gpp(\zeta\pm)\cdot e_1&=\lim_{\e\searrow0}\frac{\gpp(\zeta\pm\e e_1)-\gpp(\zeta)}{\pm\e} \\  \text{and}\quad 
\nabla\gpp(\zeta\pm)\cdot e_2&=\lim_{\e\searrow0}\frac{\gpp(\zeta\mp\e e_2)-\gpp(\zeta)}{\mp\e}.
\end{align*}
Concavity of $\gpp$ ensures  the  limits   exist.   $\nabla\gpp(\xi\pm)$ coincide (and equal $\nabla\gpp(\xi)$) 
 if and only if $\xi\in\Diff$.  
Furthermore,  on $\ri\Uset$, 
\begin{align}
\nabla\gpp(\zeta-)=\lim_{\xi\cdot e_1\nearrow\zeta\cdot e_1} \nabla\gpp(\xi\pm),
\quad
\nabla\gpp(\zeta+)=\lim_{\xi\cdot e_1\searrow\zeta\cdot e_1} \nabla\gpp(\xi\pm), \label{nabla-g-lim}
\end{align}
\begin{align}
\text{and}\quad\gpp(\zeta)= \nabla\gpp(\zeta\pm)\cdot\zeta.\label{euler}
\end{align}

For $\xi\in\ri\Uset$ define maximal  line segments on which $\gpp$ is linear,  $\Uset_{\xi-}$ for the left gradient at $\xi$ 
and $\Uset_{\xi+}$ for the right gradient at $\xi$,  
by 
	\begin{align}\label{eq:sector1}
\Uset_{\xi\pm}=\{\zeta\in\ri\Uset: 	\gpp(\zeta)-\gpp(\xi)=\nabla g(\xi\pm)\cdot(\zeta-\xi)\}.  
	\end{align}
Either or both segments can degenerate to a point. 
Let \be\label{eq:sector2} \Uset_\xi=\Uset_{\xi-}\cup\,\Uset_{\xi+}=[\ximin, \ximax]
\qquad\text{with $\ximin\cdot e_1\le \ximax\cdot e_1$.}
\ee
 If $\xi\in\Diff$ then $\Uset_{\xi+}=\Uset_{\xi-}=\Uset_\xi$, while if $\xi\notin\Diff$ then $\Uset_{\xi+}\cap\Uset_{\xi-}=\{\xi\}$.  If $\xi\in\EP$ then $\Uset_{\xi}= \{\xi\}$.     
    Notations $\ximin$ and $\ximax$ can be iterated:    $\underline\ximin=\etamin$ for  $\eta=\ximin$ and  $\overline\ximax=\zetamax$ for  $\zeta=\ximax$.
  If $\ximin\in\Diff$ then 
  $\underline\ximin=\ximin$ and similarly for  $\overline\ximax$.      When needed we use the convention   $\Uset_{e_i}=\Uset_{e_i\pm}=\{e_i\}$, $i\in\{1,2\}$.
  
For $\zeta\cdot e_1<\eta\cdot e_1$ in $\ri\Uset$,  $[\zeta, \eta]$ is a     
  {\sl maximal linear segment} for  $\gpp$   if $\nabla\gpp$ exists and  is constant in $]\zeta, \eta[$  but not on any strictly larger open line segment in $\ri\Uset$.  
Then    $[\zeta, \eta]=\Uset_{\zeta+}=\Uset_{\eta-}=\Uset_\xi$ for any $\xi\in\;]\zeta, \eta[$.   If furthermore  $\zeta, \eta\in\Diff$ we say that $\gpp$ is differentiable at the endpoints of this maximal linear segment. This hypothesis appears  several times.    A linear segment  of $\gpp$  must lie in the interior $\ri\Uset$.   This is a consequence of Martin's  shape universality  on the boundary of $\R_+^2$  \cite[Theorem 2.4]{Mar-04} which states  that 
		\begin{align}
		\label{eq:g-asym}
			\gpp(1,s)=\E(\w_0)+2\sigma\sqrt{ s}+o(\sqrt{s}\,)  \quad \text{as} \;  s\searrow 0.  
		\end{align}
  $\gpp$ is {\sl strictly concave}  if   there is no nondegenerate line segment on $\ri\Uset$ on which $\gpp$ is   linear. 

Exposed points can be characterized as follows.  All points of $(\ri\Uset)\smallsetminus\Diff$ are exposed.  A point $\xi\in\Diff$ is exposed if and only if it does not lie in any closed linear segment of $\gpp$.

\smallskip	
	
\subsection{Geodesics}\label{subsec:geo} 	


For $u\le v$ in $\Z^2$ an admissible path $x_{0,n}$ from  $x_0 = u$ to  $x_n=v$  (with $n=\abs{v-u}_1$)  
is   a (finite) {\sl geodesic} from $u$ to $v$ 
if 
	\[
		\Gpp_{u,v} = \sum_{k=0}^{n-1} \w_{x_k}.
	\]
An infinite up-right path $x_{0,\infty}=(x_k)_{0\le k<\infty}$ is a {\sl semi-infinite  geodesic} emanating from $u\in\Z^2$ if $x_0=u$ and for all $j>i\ge0$, $x_{i,j}$ is a geodesic between $x_i$ and $x_j$. 
Two semi-infinite geodesics $x_{0, \infty}$ and $y_{0, \infty}$ {\sl coalesce} if there exist $m,n\in\Z_+$ such that  $x_{m+i}=y_{n+i}$ $\forall i\in\Z_+$. 


For $\xi\in\Uset$, a  geodesic $x_{0,\infty}$ is $\xi$-directed  or  a $\xi$-geodesic if $x_n/\abs{x_n}_1\to\xi$ as $n\to\infty$.    A       directed geodesic  is $\xi$-directed for some $\xi$.   Flat segments of $\gpp$  on $\Uset$ prevent us from asserting that all geodesics are directed.  Hence we say more generally for  a subset $\cV\subset\Uset$ that  a geodesic $x_{0,\infty}$ is $\cV$-directed 
if all the limit points of 
$x_n/\abs{x_n}_1$ lie in $\mathcal V$.        Recall   the definition of $\Uset_{\xi\pm}$ from \eqref{eq:sector1}  and  $\Uset_{\xi}=\Uset_{\xi+}\cup\Uset_{\xi-}$.

\begin{theorem}\noindent\label{thm1}
\begin{enumerate}[\ \ \rm(i)]	
	\item\label{thm1:exist}   The following statements hold for $\P$-almost every $\w$.  For every $u\in\Z^2$ and  $\xi\in\Uset$ there exists at least one   semi-infinite $\Uset_{\xi+}$-directed geodesic  and at least one   semi-infinite $\Uset_{\xi-}$-directed geodesic starting from $u$.  Every 
	semi-infinite geodesic   is $\Uset_\xi$-directed for some $\xi\in\Uset$.
\item\label{thm1:direct}	If $\gpp$ is strictly concave   then,  	
		$\P$-almost surely, every semi-infinite geodesic is directed.
\item\label{thm1:cont}Suppose  $\P\{\w_0\le r\}$ is a continuous function of $r\in\R$.
Fix $\xi\in\ri\Uset$ with  $\Uset_\xi=[\,\ximin,\ximax\,]$ satisfying $\ximin, \xi,\ximax\in\Diff$.  Then   
		  $\P$-almost surely  there is a unique $\Uset_\xi$-directed  semi-infinite  geodesic out of every $u\in\Z^2$  and all these geodesics coalesce.  
		  For each $u\in\Z^2$ there are at most finitely many sites $v\in\Z^2$ such that the unique $\Uset_\xi$-directed  semi-infinite geodesic out of $v$ goes through $u$.
\end{enumerate}
\end{theorem}

Under the hypotheses of part (iii) we have   $\Uset_{\xi\pm}=\Uset_\xi$.  So there is no contradiction  between parts (i) and (iii).  

By   \eqref{eq:g-asym} there are infinitely many distinct sets $\Uset_{\xi\pm}$. 
Hence,  without any assumptions on 
the shape $\gpp$,  part \eqref{thm1:exist} implies the existence of infinitely many semi-infinite geodesics from each point $u\in\Z^2$. 
The second part of claim \eqref{thm1:cont} prevents the existence of doubly infinite geodesics $x_{-\infty,\infty}$ such that $x_{0,\infty}$ is $\Uset_\xi$-directed (a.s.\ in a fixed direction $\xi$).  This  is not true for all weight distributions (see Example \ref{ex:flat} below).  

 For  exponentially distributed $\w_0$ the  results of Theorem \ref{thm1} appeared earlier as follows.     Theorem \ref{thm1}\eqref{thm1:exist}--\eqref{thm1:direct} is covered by  Proposition 7 of \cite{Fer-Pim-05}.    Uniqueness and coalescence in  part \eqref{thm1:cont} are in   Theorem 1(3) of \cite{Cou-11}, combined with the coalescence proof of \cite{Lic-New-96} which was adapted to exponential LPP in Proposition 8  of   \cite{Fer-Pim-05}.   Nonexistence of doubly infinite geodesics is part of Lemma 2 of \cite{Pim-13-}.   

When the  distribution of  $\w_0$ is  not   continuous,  uniqueness of geodesics (Theorem \ref{thm1}\eqref{thm1:cont}) cannot hold. Then we can consider leftmost and rightmost geodesics.  The  {\sl leftmost} geodesic $\underline{x}_{\;\centerdot}$ (between two given points or in a given direction)  satisfies 
	$	\underline{x} _k \cdot e_1 \le x _k \cdot e_1 $ 
for any geodesic $x_\centerdot$  of the same category.   The rightmost geodesic satisfies the opposite inequality.

\begin{theorem}\label{thm:lr-geod}
Fix   $\xi\in\ri\Uset$.  
The following hold almost surely.
\begin{enumerate}[\ \ \rm(i)]
\item\label{thm:lr-geod:i} 
Assume $\ximin$ is not the right endpoint of a linear segment of $\gpp$ {\rm(}equivalently, $\underline\ximin=\ximin${\rm)}. 
Then there exists a leftmost $\Uset_{\xi-}$-directed geodesic from each $u\in\Z^2$ and 
all these leftmost geodesics coalesce.  
		  For each $u\in\Z^2$ there are at most finitely many sites $v\in\Z^2$ such that the leftmost $\Uset_{\xi-}$-directed geodesic out of $v$ goes through $u$.
A similar statement holds for rightmost $\Uset_{\xi+}$-geodesics, under the assumption $\overline\ximax=\ximax$.
\item\label{thm:lr-geod:ii}  Assume $\ximin,\xi,\ximax\in\Diff$. Then for any $u\in\Z^2$ and  sequence $v_n$ such that
\begin{align}\label{eq:vn}
\abs{v_n}_1\to\infty\quad\text{and}\quad\ximin\cdot e_1\le\varliminf_{n\to\infty}\frac{v_n\cdot e_1}{\abs{v_n}_1}\le\varlimsup_{n\to\infty}\frac{v_n\cdot e_1}{\abs{v_n}_1}\le\ximax\cdot e_1,
\end{align} 
the leftmost geodesic from $u$ to $v_n$ converges to the unique  leftmost $\Uset_\xi$-directed geodesic  from $u$ given in part 
\eqref{thm:lr-geod:i}.   
A similar statement holds for  rightmost geodesics. 
\end{enumerate}
\end{theorem}

The convergence  statement  Theorem \ref{thm:lr-geod}\eqref{thm:lr-geod:ii} applies also to the case in Theorem \ref{thm1}\eqref{thm1:cont},  and in that   case there is just one unique $\Uset_\xi$-directed geodesic, not separate leftmost and rightmost geodesics.  
Theorems \ref{thm1} and \ref{thm:lr-geod} are proved in Section \ref{sec:geod}. In particular,  we give explicit local recipes in terms of a priori constructed  cocycles  for defining the geodesics whose existence is claimed in the theorems. 

\smallskip

\subsection{Busemann functions and Busemann geodesics} 



By 
  \eqref{Gxy1}   the following identities hold  along any geodesic $x_{0,m}$ from $u$ to $v_n$: 
	\be\label{w-geod}\begin{aligned}
	\w_{x_i}&=\min\bigl(\Gpp_{x_i,v_n}-\Gpp_{x_i+e_1,v_n}\,,\,\Gpp_{x_i,v_n}-\Gpp_{x_i+e_2,v_n}\bigr)\\
	&=\Gpp_{x_i,v_n}-\Gpp_{x_{i+1},v_n},\quad \text{ for } \ \ \ 0\le i<m.
	\end{aligned} \ee
The  second equality in \eqref{w-geod} shows how to construct a finite geodesic ending at $v_n$. To study semi-infinite geodesics we take $v_n\to\infty$ in a particular direction.
Point-to-point {\sl  Busemann functions}  are limits  of gradients $\Gpp_{x, v_n}-\Gpp_{y,v_n}$.
The next  existence theorem  is Theorem~3.1 from  \cite{Geo-Ras-Sep-15a-}.

\begin{theorem}\label{thm:buse}  
 Fix two points $\zeta, \eta\in\Diff$ such that $\zeta\cdot e_1\le\eta\cdot e_1$.  Assume that  either 
 
 {\rm(i)}  $\zeta=\eta=\xi\in\EP$ in which case  $\zeta=\eta=\xi=\ximin= \ximax$,   or that 
 
 {\rm(ii)}  $[\zeta, \eta]$ is a maximal linear segment of $\gpp$ in which case  
 $[\zeta, \eta]=[\ximin, \ximax]$ for all $\xi\in[\zeta, \eta]$.  
 
Then there exist integrable random variables 
$\{B(x,y):x,y\in\Z^2\}$ and an event  $\Omega_0$ with $\P(\Omega_0)=1$ such that  the following holds for each 
$\w\in\Omega_0$.    For each sequence $v_n\in\Z_+^2$ such that 
\begin{align}\label{eq:vn15}
\abs{v_n}_1\to\infty\quad\text{and}\quad\zeta\cdot e_1\le\varliminf_{n\to\infty}\frac{v_n\cdot e_1}{\abs{v_n}_1}\le\varlimsup_{n\to\infty}\frac{v_n\cdot e_1}{\abs{v_n}_1}\le\eta\cdot e_1,
\end{align}
we have the limit
\be \label{eq:grad:coc1} 
							B(\w, x,y) = \lim_{n\to \infty} \big( \Gpp_{x, v_n}(\w) - \Gpp_{y, v_n}(\w) \big)   \qquad  \text{ for  $x,y\in\Z^2$. } 
						\ee 
 The mean of the limit is given by 
\be\label{EB=Dg}     \nabla \gpp(\xi) = \bigl( \, \E[B(x,x+e_1)]\,,\,   \E[B(x,x+e_2)] \,\bigr)\quad\text{for all $\xi\in[\zeta, \eta]$}.
\ee 
\end{theorem}
 
In particular,   suppose  $\xi$ is an exposed point of differentiability of $\gpp$,  or $\xi$ lies on a maximal linear segment of $\gpp$ whose endpoints are points of differentiability.  Then a Busemann function $B^\xi$ exists in direction $\xi$ in the sense that   $B^\xi(\w,x,y)$ equals  the limit  in \eqref{eq:grad:coc1}   for any sequence $v_n/\abs{v_n}_1\to\xi$ with $\abs{v_n}_1\to\infty$.    Furthermore, the $B^\xi$'s match on  maximal linear segments of $\gpp$ with endpoints in $\Diff$.

  Limit \eqref{eq:grad:coc1} applied to  \eqref{w-geod} gives 
\be\label{var151}\w_{x_i}=  \min_{j\in\{1,2\}} B(\w,x_i,x_i+e_j)=B(\w,x_i,x_{i+1})\quad\P\text{-a.s.}\ee 
The  second equality shows how to construct semi-infinite  geodesics from a Busemann function. Such geodesics will be called  {\sl Busemann geodesics}.
The next theorem   shows that in a direction that satisfies the differentiability assumptions that ensure existence of Busemann functions,  all geodesics  are Busemann geodesics.   
 
\begin{theorem}\label{thm:buse-geo}
Fix   $\xi\in\ri\Uset$ with  $\Uset_\xi=[\ximin,\ximax]$ such that  $\ximin, \xi, \ximax\in\Diff$.    Let $B$ be the limit from \eqref{eq:grad:coc1}. Then there exists  an event  $\Omega_0$ with $\P(\Omega_0)=1$ and such that  statements \eqref{thm:buse-geo:i}--\eqref{thm:buse-geo:iii} below  hold for each 
$\w\in\Omega_0$.  
\begin{enumerate}[\ \ \rm(i)]
\item\label{thm:buse-geo:i} Every up-right path $x_{0,\infty}$ such that $\w_{x_k}=B(x_k,x_{k+1})$ for all $k\ge0$
is a semi-infinite geodesic. We call such a path a Busemann geodesic for $B$.
\item\label{thm:buse-geo:ii} Every semi-infinite geodesic $x_{0,\infty}$ that satisfies 
	\be\label{geod-98} \ximin\cdot e_1\le\varliminf_{n\to\infty}\frac{x_n\cdot e_1}n\le\varlimsup_{n\to\infty}\frac{x_n\cdot e_1}n\le\ximax\cdot e_1\ee
 is a Busemann geodesic for $B$.
\item\label{thm:buse-geo:iii}   Let $v_n$ be a sequence that satisfies \eqref{eq:vn}.  Let $m\in\N$.  Then $\exists n_0\in\N$ such that if $n\ge n_0$, then every geodesic $x_{0,\abs{v_n}_1}$ from  $x_0=0$ to $v_n$ satisfies  $B(\w,x_i,x_{i+1})=\w_{x_i}$ for all $i=0,1, \dotsc, m$. 
\end{enumerate}
\end{theorem}

Note in particular that the unique geodesics discussed in  Theorem \ref{thm1}\eqref{thm1:cont} and Theorem \ref{thm:lr-geod}\eqref{thm:lr-geod:ii}  are Busemann geodesics.  
This theorem is proved in Section \ref{sec:geod}.

\begin{example}[Flat edge in the percolation cone]   \label{ex:flat}  Assume \eqref{2d-ass} and furthermore that $\w_0\le 1$ and  $\opc<\P\{\w_0=1\}<1$ where $\opc$ is the critical probability of oriented site percolation  on $\Z^2$ (see Section~3.2 of \cite{Geo-Ras-Sep-15a-} for more detail about this setting).    Then $\gpp$ has  a nondegenerate, symmetric  linear segment   $[\etamin, \etamax]$ such that $\etamin, \etamax\in\Diff$  
 \cite{Auf-Dam-13,Dur-Lig-81,Mar-02}.    According to Theorems \ref{thm:lr-geod} and  \ref{thm:buse-geo},  from any point $u\in\Z^2$ there exist unique leftmost  and rightmost  semi-infinite geodesics directed into the segment  $[\etamin, \etamax]$, these geodesics are Busemann geodesics, and finite leftmost and rightmost geodesics converge to these Busemann geodesics.  
 
 Note also, in relation to Theorem \ref{thm1}\eqref{thm1:cont}, that a doubly infinite geodesic through the origin  with $\w_{x_k}\equiv 1$, directed into $[\etamin, \etamax]$,  can be constructed with positive probability  by joining together a percolating path in the first quadrant with one in the   third quadrant.     

\end{example}

\smallskip

\subsection{Competition interface}  \label{sec:ci} 

For this subsection assume that  $\P\{\w_0\le r\}$ is a continuous function of $r\in\R$.  Then  with probability one  no two finite paths of any lengths have equal weight.  Consequently for  any $v\in\Z_+^2$ there is a unique finite geodesic between $0$ and $v$.  Together these finite geodesics form the {\sl geodesic tree} $\cT_0$ rooted at $0$  that spans $\Z_+^2$.   The two subtrees 
rooted at $e_1$ and $e_2$ are separated by an up-right path $\varphi=(\varphi_k)_{k\ge 0}$  on the dual  lattice $(\frac12,\frac12)+\Z_+^2$ with $\varphi_0=(\frac12,\frac12)$.    The path $\varphi$ is called the {\sl competition interface}.  
The term  comes from the interpretation that the subtrees 
 are two competing   infections 
on the lattice \cite{Fer-Mar-Pim-09,Fer-Pim-05}.  See Figure \ref{cif-fig}.  

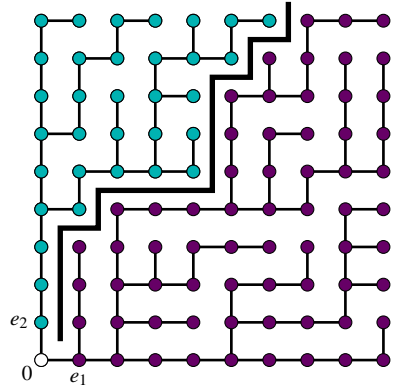
\begin{figure}[h]
\begin{center}

\begin{tikzpicture}[>=latex,scale=0.5]

\definecolor{sussexg}{rgb}{0,0.7,0.7}
\definecolor{sussexp}{rgb}{0.4,0,0.4}

\draw(-0.37,-0.35)node{$0$};
\draw(-0.57,1)node{$e_2$};
\draw(1,-0.5)node{$e_1$};


\draw[line width=1pt](9,0)--(0,0)--(0,9);     
\draw[line width=1pt](2,0)--(2,4)--(5,4)--(5,7)--(7,7)--(7,9)--(9,9);
\draw[line width=1pt](1,0)--(1,3);
\draw[line width=1pt](0,4)--(1,4)--(1,5)--(2,5)--(2,7);
\draw[line width=1pt](2,5)--(4,5)--(4,6);
\draw[line width=1pt](3,5)--(3,6)--(3,8)--(5,8)--(5,9)--(6,9);
\draw[line width=1pt](3,7)--(4,7);
\draw[line width=1pt](6,7)--(6,8);
\draw[line width=1pt](0,6)--(1,6)--(1,8)--(2,8)--(2,9)--(3,9);
\draw[line width=1pt](4,8)--(4,9);
\draw[line width=1pt](0,9)--(1,9);
\draw[line width=1pt](5,4)--(7,4)--(7,5)--(9,5)--(9,8);
\draw[line width=1pt](6,4)--(6,6)--(7,6);
\draw[line width=1pt](8,5)--(8,8);
\draw[line width=1pt](2,1)--(8,1)--(8,2)--(9,2);
\draw[line width=1pt](2,2)--(4,2)--(4,3)--(6,3);
\draw[line width=1pt](9,0)--(9,1);

\draw[line width=1pt](5,0)--(5,1);
\draw[line width=2pt,color=white](4,1)--(5,1);

\draw[line width=1pt](5,1)--(5,2)--(7,2)--(7,3);
\draw[line width=1pt](8,2)--(8,4)--(9,4);
\draw[line width=1pt](3,2)--(3,3);
\draw[line width=1pt](8,3)--(9,3);


\draw[line width=2.0pt](.5,.5)--(0.5,3.5)--(1.5,3.5)--(1.5, 4.5)--(4.5,4.5)--(4.5,7.5)--(5.5,7.5)--(5.5,8.5)--(6.5,8.5)--(6.5, 9.5);   
\foreach \x in {0,...,9}{
             \foreach \y in {0,...,9}{
\draw[ fill =sussexp](\x,\y)circle(1.7mm);    
}}

\foreach \x in{0,...,6}{
		\draw[ fill =sussexg](\x, 9)circle(1.7mm);    
				}

\foreach \x in{0,...,5}{
		\draw[ fill =sussexg](\x, 8)circle(1.7mm);    
				}

\foreach \x in{0,...,4}{
		\foreach \y in {5,6,7}{
		\draw[fill =sussexg](\x, \y)circle(1.7mm);    
				}}

\foreach \x in{0,...,1}{
		\draw[ fill =sussexg](\x, 4)circle(1.7mm);    
				}

\foreach \x in{0,...,0}{
		\foreach \y in {1,2,3}{
		\draw[fill =sussexg](\x, \y)circle(1.7mm);    
				}}

\draw[fill=white](0,0)circle(1.7mm); 

\end{tikzpicture}

\end{center}
\caption{\small   The geodesic tree $\mathcal{T}_0$ rooted at $0$.  
The competition interface  (solid line)  emanates from $(\tfrac12, \tfrac12)$ and separates 
the   subtrees of $\mathcal{T}_0$ rooted at $e_1$   and $e_2$.}
\label{cif-fig}\end{figure}

 Adopt  the convention that $\Gpp_{e_i, ne_j}=-\infty$ for $i\ne j$ and $n\ge 0$ (there is  no admissible  path from $e_i$ to $ ne_j$).    Fix $n\in\N$.   As $v$ moves to the right with  $\abs{v}_1=n$ fixed,  the function $\Gpp_{e_2,v}-\Gpp_{e_1,v}$ is nonincreasing (Lemma \ref{lm:new:comp0}  in the appendix). Then   $\varphi_{n-1}=(k+\tfrac12,n-k-\tfrac12)$ for the  unique $0\le k<n$ such that   
 \be\label{ci:19} \Gpp_{e_2,(k,n-k)}-\Gpp_{e_1,(k,n-k)}>0> \Gpp_{e_2,(k+1,n-k-1)}-\Gpp_{e_1,(k+1,n-k-1)}. \ee  
 
 \begin{theorem}\label{thm:ci-1} 
Assume $\P\{\w_0\le r\}$ is continuous in $r$ and that $\gpp$ is differentiable at the endpoints of all   its linear segments.
 Then we have the law of large numbers  
\begin{align}\label{ci-lln2}
\cid(\w)= \lim_{n\to\infty} n^{-1}\varphi_n(\w)\qquad \text{$\P$-a.s.}   \end{align}
The   limit $\cid$ is almost surely an exposed point in $\ri\Uset$  {\rm(}recall definition \eqref{eq:epod}{\rm)}.  For any $\xi\in\ri\Uset$,   $\P(\cid=\xi)>0$ if and only if $\xi\in(\ri\Uset)\smallsetminus\Diff$.    Any open interval outside  the closed  linear segments of   $\gpp$  contains $\cid$ with positive probability.  
\end{theorem}

When $\w_0$ has continuous distribution,  $\gpp$ is expected to be strictly concave.   Thus the assumption that $\gpp$ is differentiable at the endpoints of its linear segments  should really be vacuously true in the theorem.   In light of the expectation that $\gpp$ is differentiable, the conjecture for  $\cid$ would be that it has a continuous distribution.  

In the exponential case, \eqref{ci-lln2}  and the explicit distribution of $\cid$ were given  in Theorem 1 of \cite{Fer-Pim-05}.


\begin{remark}\label{rem:cif-var}    Assume  that  $\P\{\w_0\le r\}$ is continuous and  that $\gpp$ is either differentiable or strictly concave on 
$\ri\Uset$ so that no caveats are needed.  The minimum   in  \eqref{var151}  with $B=B^\xi$ is taken at $j=1$ if $\xi\cdot e_1>\cid(T_{x_i}\w)\cdot e_1$ and  at $j=2$ if $\xi\cdot e_1<\cid(T_{x_i}\w)\cdot e_1$.   
This  will become clear  from an alternative definition  \eqref{ci} of  $\cid$.  
\end{remark}

The competition interface is a natural direction in which there are two geodesics out of $0$.  Nonuniqueness in the random direction $\cid$ does not violate the almost sure uniqueness in a fixed direction given in Theorem \ref{thm1}\eqref{thm1:cont}.       For $x\in\Z^2$ let $\Uset_*^x$ be the random set of directions $\xi\in\Uset$ such that there are at least two $\Uset_\xi$-directed semi-infinite geodesics out of $x$.  

 \begin{theorem}\label{thm:ci-2} 
Assume $\P\{\w_0\le r\}$ is continuous in $r$. Assume $\gpp$ is differentiable at the endpoints of all its linear segments. The following statements are true with $\P$-probability one and for all $x\in\Z^2$.
\begin{enumerate}[\ \ \rm(i)]
\item\label{thm:ci-2:(i)} $\cid(T_x\w)$ is the unique direction $\xi$ such that there are at least two $\Uset_\xi$-directed semi-infinite geodesics from $x$ that separate at $x$ and never intersect thereafter.
\item\label{thm:ci-2:(ii)}     $\Uset_*^x$ contains all $\xi\in(\ri\Uset)\smallsetminus\Diff$,  intersects every open interval outside the closed linear segments of $\gpp$, and is a countably infinite subset of $\{\cid(T_z\w):z\ge x\}$. 
\end{enumerate} 
\end{theorem}

In the exponential case Theorem 1(1)--(2)  of \cite{Cou-11} showed   that $\Uset_*^x$ is countably infinite and dense. 
  
 
Theorems \ref{thm:ci-1} and  \ref{thm:ci-2} are  proved in Section \ref{sec:ci-pf}.    More is actually true.  In Section \ref{sec:ci-pf} we define $\cid$  on a larger probability space  in terms of a priori constructed cocycles,  without any assumptions on $\gpp$. 
   Then even  without  the differentiability assumptions of  Theorems \ref{thm:ci-1} and \ref{thm:ci-2}, corners of the limit shape are the  atoms of $\cid$, and 
    there are at least two $\Uset_{\cid\circ T_x}$-directed semi-infinite geodesics out of $x$ that immediately separate and never intersect after that.  
   (See Theorem \ref{th:two-geo} below.)    
 

\smallskip 

When   $\w_0$ does  not have continuous distribution, there are two competition interfaces: one  for the  tree of leftmost geodesics   and one 
for the tree of  rightmost geodesics.    Then $\cid$ has natural left and right versions, defined  in \eqref{ci12}.  
We compute the limit distributions of the two competition interfaces for geometric weights in Sections \ref{subsec:solv} and \ref{sec:solv}. 


\smallskip

\subsection{Exactly solvable models}\label{subsec:solv}
We illustrate our results in the two exactly solvable cases: the distribution of the  weights   $\w_x$ with mean $\Ew>0$  is  either  
\be\label{cases7} \begin{aligned}
\text{exponential:  } \P\{\w_x\ge t\}&=  e^{-t/\Ew} \text{ for  $t\in\R_+$ with $\sigma^2=\Ew^2$,   }\\
\text{or geometric: } \P\{\w_x\ge k\}&=(1-\Ew^{-1})^{k} \text{ for  $k\in\N$ with  $\sigma^2=\Ew(\Ew-1)$.}
\end{aligned}\ee

For both cases in \eqref{cases7}  the point-to-point limit   function is 
 \[\gpp(\xi)=\Ew(\xi\cdot e_1+\xi\cdot e_2)+2\sigma\sqrt{(\xi\cdot e_1)(\xi\cdot e_2)}\,.\]
 In the exponential case this formula was first derived by Rost \cite{Ros-81} (who presented the model in its coupling with TASEP without the last-passage formulation)  while early derivations of the geometric case appeared in  \cite{Coh-Elk-Pro-96, Joc-Pro-Sho-98, Sep-98-mprf-1}.  

Since   $\gpp$  is  differentiable and strictly concave, $\ri\Uset=\EP$
and all the results of the previous sections are valid.   
 Theorem \ref{thm:buse} implies that Busemann functions \eqref{eq:grad:coc1} exist in all directions  $\xi\in\ri\Uset$.
The probability distribution of $B^\xi$ can be described explicitly. For the exponential case see for example Theorem 8.1 in \cite{Cat-Pim-12} or Section 3.3 in \cite{Cat-Pim-13},  and  Sections~3.1 and 7.1 in \cite{Geo-Ras-Sep-15a-} for both cases.  

  Section \ref{subsec:geo} gives the following results on geodesics.
For almost every $\w$ every semi-infinite geodesic has a direction. For every fixed direction $\xi\in\ri\Uset$ the following holds almost surely. There exists a $\xi$-geodesic out of every lattice point. 
In the exponential case, these $\xi$-geodesics are unique and coalesce. In the geometric case uniqueness cannot hold, but there exists 
a unique leftmost $\xi$-geodesic out of each lattice point and
these leftmost $\xi$-geodesics coalesce. The same holds for rightmost $\xi$-geodesics. Finite (leftmost/rightmost) geodesics from $u\in\Z^2$ to $v_n$ converge to infinite (leftmost/rightmost) 
$\xi$-geodesics out of $u$, as $v_n/\abs{v_n}_1\to\xi$ with $\abs{v_n}_1\to\infty$. 


 The description of random directions for nonuniqueness of geodesics in Theorem \ref{thm:ci-2}\eqref{thm:ci-2:(i)}--\eqref{thm:ci-2:(ii)} applies to the exponential case. 
In the exponential case the asymptotic direction   $\cid$ of the competition interface  given by Theorem \ref{thm:ci-1}   has been studied by several authors, not only for percolation in the first quadrant $\Z_+^2$ as studied here,  but with much more general  initial profiles 
\cite{Fer-Pim-05, Fer-Mar-Pim-09, Cat-Pim-13}.   



The model with  geometric weights has a tree of leftmost geodesics with competition interface $\varphi^{(l)}=(\varphi^{(l)}_k)_{k\ge 0}$  and a tree of  rightmost geodesics with competition interface $\varphi^{(r)}=(\varphi^{(r)}_k)_{k\ge 0}$.  
  Note that  $\varphi^{(r)}$ is to the {\sl left} of  $\varphi^{(l)}$  because in \eqref{ci:19} there is now a middle range   $\Gpp_{e_2,(k,n-k)}-\Gpp_{e_1,(k,n-k)}=0$ that is to the right (left) of  $\varphi^{(r)}$ ($\varphi^{(l)}$).  Strict concavity of the limit $\gpp$ implies (with the arguments of Section \ref{sec:ci-pf})  the almost sure  limits 
\[  n^{-1}  \varphi^{(l)}_n \to \cid^{(l)} \quad \text{and}\quad 
n^{-1}  \varphi^{(r)}_n \to \cid^{(r)}. \]
The angles $\theta_*^{(a)} =\tan^{-1} ( \cid^{(a)}\!\cdot e_2/\cid^{(a)}\!\cdot e_1)$  for $a\in\{l,r\}$  have the following  distributions (with $p_0=m_0^{-1}$ denoting the success probability of the geometric): for  $t\in[0,\pi/2]$ 
\be\label{geom:theta} \begin{aligned} 
\P\{\theta_*^{(r)}\le t\}&=\frac{\sqrt{(1-p_0)\sin t}}{\sqrt{(1-p_0)\sin t}+\sqrt{\cos t}}\,  \\
\text{and}\qquad  \P\{\theta_*^{(l)}\le t\}&=\frac{\sqrt{\sin t}}{\sqrt{\sin t}+\sqrt{(1-p_0)\cos t}}\, .\end{aligned}\ee
  Section \ref{sec:solv} derives \eqref{geom:theta}.  Taking $p_0\to0$ recovers   the exponential case first proved in \cite{Fer-Pim-05}.  
\bigskip
 
We turn to describe the setting of stationary cocycles in which  our results are  proved.

\section{Stationary cocycles and Busemann functions}\label{sec:cocycles}

The results of this paper are based on a construction of stationary cocycles on an extended space  $\Ombig=\Omega\times\Omega'$ where $\Omega=\R^{\Z^2}$ is the original environment space and $\Omega'=S^{\Z^2}$ is another Polish product space.    The details of the construction are in Section~7 of \cite{Geo-Ras-Sep-15a-}.  Spatial translations act   in the usual manner:   with  generic elements  of $\Ombig$  denoted by 
$\what=(\w, \w') =(\w_x, \w'_x)_{x\in\Z^2}=(\what_x)_{x\in\Z^2}$,  $(T_x\what)_y=\what_{x+y}$ for $x,y\in\Z^2$.    The extended probability space is  $(\Ombig, \kSbig, \Pbig)$  where $\kSbig$ is the Borel $\sigma$-algebra and $\Pbig$ is a translation-invariant probability measure. $\Ebig$ denotes expectation under $\Pbig$.    In this setting a cocycle is defined as follows.

\begin{definition} 
\label{def:cK}
	A measurable function $B:\Ombig\times\Z^2\times\Z^2\to\R$ is   a {\rm stationary $L^1(\Pbig)$  cocycle}
	if it satisfies the following three conditions $\forall x,y,z\in\Z^2$. 
		\begin{enumerate}[\ \ \rm(a)]
			\item\label{def:cK:int} Integrability:   $\Ebig\abs{B(x,y)} <\infty$.
			\item\label{def:cK:stat} Stationarity:  for $\Pbig$-a.e.\ $\what$, 
					$B(\what, z+x,z+y)=B(T_z\what, x,y)$.
			\item\label{def:cK:coc} Additivity:   for $\Pbig$-a.e.\ $\what$,  $B(\what, x,y)+B(\what, y,z)=B(\what, x,z)$.
		\end{enumerate}
\end{definition}

    The cocycles of interest are related to the last-passage weights through the next definition.
	\begin{definition}
		\label{def:bdry-model}
			A stationary $L^1$ cocycle $B$ on $\Omhat$  {\rm recovers} 
			weights $\w$ if
				\be \label{eq:VB}
					\w_x=\min_{i\in\{1,2\}} B(\what, x, x+e_i)\quad\text{ for }\Pbig\text{-a.e.\ }\what\ \text{ and } \ \forall x\in\Z^2.
				\ee
	 \end{definition}

The next theorem  (reproduced from Theorem~5.2 in \cite{Geo-Ras-Sep-15a-}) states  the existence  and  properties of the cocycles.   
Assumption  \eqref{2d-ass} is in force.   This is the only place where the assumption $\P(\w_0\ge c)=1$ is needed,  and the only reason is that the queueing results that are used to prove the theorem assume  $\w_0\ge 0$.  
In part \eqref{th:cocycles:indep} below we use  this notation:  for  a finite or infinite set $I\subset\Z^2$,   $I^<=\{x\in\Z^2: x\not\ge z \;\forall z\in I\}$ is the set of lattice points that
do not lie on a  ray from $I$ at an angle in $[0,\pi/2]$. For example, if $I=\{0,\dotsc,m\}\times\{0,\dotsc,n\}$ then $I^<=\Z^2\smallsetminus\Z_+^2$.  
 
\begin{theorem}\label{th:cocycles}
There exist real-valued Borel functions 
$B^\xi_\pm(\what,x,y)$ of  $(\what,\xi,x,y)\in\Ombig\times \ri\Uset\times\Z^2\times\Z^2$   and a translation invariant 
  Borel probability measure $\Pbig$ on  $(\Ombig, \kSbig)$ such that  the following  properties hold.  
 \begin{enumerate}[\ \ \rm(i)]
\item\label{th:cocycles:indep} 
Under $\Pbig$, the marginal distribution of the configuration $\w$ is the i.i.d.\ measure $\P$ specified in assumption \eqref{2d-ass}. 
For each $\xi\in\ri\Uset$ and $\pm$, the  $\R^3$-valued  process   
$\{\psi^{\pm, \xi}_x\}_{x\in\Z^2}$
defined by   
\begin{align}\label{erg-B}
\psi^{\pm, \xi}_x(\what)=(\w_x, B^\xi_{\pm}(\what, x,x+e_1), B^\xi_{\pm}(\what,x,x+e_2))
\end{align}
is separately ergodic under both translations $T_{e_1}$ and $T_{e_2}$.  
For any $I\subset\Z^2$,   the variables 
\[ \bigl \{ (\w_x, B^\xi_{+}(\what,x,x+e_i), B^\xi_{-}(\what,x,x+e_i)) : x\in I, \, \xi\in\ri\Uset, \, i\in\{1,2\} \bigr\}\] 
 are  independent of $\{ \w_x : x\in I^< \}$.   \\[-8pt]
 
 \item\label{th:cocycles:exist}  Each process 
 $B^\xi_{\pm}=\{B^\xi_{\pm}(x,y)\}_{x,y\in\Z^2}$ 
  is a stationary $L^1(\Pbig)$ cocycle {\rm(}Definition \ref{def:cK}{\rm)}  that recovers the weights $\w$   {\rm(}Definition \ref{def:bdry-model}{\rm)}:
  \be\label{Bw-9}  \w_x=B^\xi_{\pm}(\what, x,x+e_1) \wedge B^\xi_{\pm}(\what,x,x+e_2)\qquad \text{$\Pbig$-a.s.}  \ee 
  The mean  vectors   
  satisfy
					\begin{align}\label{eq:h=grad} 
					\Ebig[B^\xi_{\pm}(0,e_1)]e_1+\Ebig[B^\xi_{\pm}(0,e_2)]e_2=\nabla\gpp(\xi\pm).
					\end{align} \\[-8pt]

\item\label{th:cocycles:flat} 			 No two distinct cocycles have  a common mean vector.  That is, if $\nabla\gpp(\xi+)=\nabla\gpp(\zeta-)$ then 
	\[  B^\xi_{+}(\what,x,y)=B^\zeta_{-}(\what,x,y)  
	\quad \forall\, \what\in\Ombig, \, x,y\in\Z^2  \] 
and similarly for all four combinations of $\pm$ and $\xi,\zeta$.  
 These equalities  hold    for all $\what$   without an   almost sure modifier    because they come directly   from the construction.   In particular,  if   $\xi\in\Diff$ then 
	\[  	 B^\xi_{+}(\what,x,y)= B^\xi_{-}(\what,x,y) = B^\xi(\what,x,y)\quad \forall \,\what\in\Ombig, \, x,y\in\Z^2, 
	 \]
	where the second equality defines the cocycle $B^\xi$.  \\[-8pt]
					
\item\label{th:cocycles:cont}  
There exists an event $\Ombig_0$ with $\Pbig(\Ombig_0)=1$ and such that {\rm(a)} and {\rm(b)} below  hold for all $\what\in\Ombig_0$, $x,y\in\Z^2$ and  $\xi,\zeta\in\ri\Uset$.  
 
{\rm(a)}  Monotonicity: if     $\xi\cdot e_1<\zeta\cdot e_1$ then  
					\be\begin{aligned}  
						B^\xi_{-}(\what, x,x+e_1) &\ge B^\xi_{+}(\what, x,x+e_1) \ge B^\zeta_{-}(\what, x,x+e_1) \\ \quad\text{and}\quad B^\xi_{-}(\what,x,x+e_2) &\le B^\xi_{+}(\what, x,x+e_2) \le B^\zeta_{-}(\what, x,x+e_2). 
					\end{aligned} \label{eq:monotone} 
					\ee
					
{\rm(b)}	 Right continuity: if $\zeta_n\cdot e_1\searrow\xi\cdot e_1$  then 
					\begin{align}\label{eq:cont}
						\lim_{n\to\infty}B^{\zeta_n}_{\pm}(\what, x,y) = B^\xi_{+}(\what,x,y)  . 
					\end{align}  \\[-8pt]
					
\item\label{th:cocycle:cont-left}
Left continuity at a fixed $\xi\in\ri\Uset$:   there exists an event $\Ombig^{(\xi)}$ with $\Pbig(\Ombig^{(\xi)})=1$ and such that for any sequence $\zeta_n\cdot e_1\nearrow\xi\cdot e_1$  
	\begin{align}\label{eq:cont-left}
						\lim_{n\to\infty}B^{\zeta_n}_{\pm}(\what, x,y) = B^\xi_{-}(\what,x,y)\qquad \text{for $\what\in\Ombig^{(\xi)}, x,y\in\Z^d$.}   
					\end{align}

\end{enumerate} 
\end{theorem}\smallskip

The  next result  (Theorem~6.1 in \cite{Geo-Ras-Sep-15a-}) relates  the cocycles $B^\xi_{\pm}$ to limiting $G$-increments.  We quote the theorem in full  for use in the proof of Theorem \ref{thm:buse-geo}  below.    \eqref{2d-ass} is  assumed.
Recall the line segment $\Uset_\xi=[\ximin, \ximax]$ with $\ximin\cdot e_1\le \ximax\cdot e_1$ from \eqref{eq:sector1}--\eqref{eq:sector2}. 

\begin{theorem}\label{th:construction} 
Fix $\xi\in\ri\Uset$.  
Then there exists an event  $\Ombig_0$ with $\Pbig(\Ombig_0)=1$ such that for each 
$\what\in\Ombig_0$ and  for any  sequence $v_n\in\Z_+^2$ that satisfies 
\begin{align}\label{eq:vn66}
\abs{v_n}_1\to\infty\quad\text{and}\quad\ximin\cdot e_1\le\varliminf_{n\to\infty}\frac{v_n\cdot e_1}{\abs{v_n}_1}\le\varlimsup_{n\to\infty}\frac{v_n\cdot e_1}{\abs{v_n}_1}\le\ximax\cdot e_1,
\end{align}
  we have
\be\label{bu:G1}  \begin{aligned}   B^{\ximax}_{+}(\what, x,x+e_1)&\le \varliminf_{n\to \infty} \big( \Gpp_{x, v_n}(\w) - \Gpp_{x+e_1, v_n}(\w) \big) \\ &\le \varlimsup_{n\to \infty} \big( \Gpp_{x, v_n}(\w) - \Gpp_{x+e_1, v_n}(\w)  \big) \le B^{\ximin}_{-}(\what, x,x+e_1)  
\end{aligned}  \ee
and 
\be\label{bu:G2}  
\begin{aligned}  B^{\ximin}_{-}(\what, x,x+e_2)  &\le \varliminf_{n\to \infty} \big( \Gpp_{x, v_n}(\w) - \Gpp_{x+e_2, v_n}(\w) \big)\\
&\le \varlimsup_{n\to \infty} \big( \Gpp_{x, v_n}(\w) - \Gpp_{x+e_2, v_n}(\w) \big) \le B^{\ximax}_{+}(\what, x,x+e_2).  
\end{aligned}  \ee
\end{theorem}
%
%
%

\begin{remark}\label{rm:cc} 
(i)   Theorem \ref{thm:buse} follows immediately because  
by Theorem \ref{th:cocycles}\eqref{th:cocycles:flat},   $B^{\ximin}_{\pm}=B^{\xi}=B^{\ximax}_{\pm}$  if  $\xi, \ximin, \ximax\in\Diff$.  
%
(ii) If $\gpp$ is assumed differentiable at the endpoints of all its linear segments,  then all cocycles $B^\xi_\pm(x,y)$ are in fact  functions of $\w$, that is, $\kS$-measurable (see Theorem~5.3 in \cite{Geo-Ras-Sep-15a-}).  
\end{remark}

\section{Directional geodesics}
\label{sec:geod}

This section proves the results on geodesics.  We define special geodesics in terms of   the cocycles $B^{\xi}_{\pm}$ from  Theorem \ref{th:cocycles},   on the extended space $\Ombig=\Omega\times\Omega'$.   Assumption \eqref{2d-ass} is in force.      The idea   is in the next  lemma, followed by  the definition of cocycle geodesics.

\begin{lemma}\label{lm:grad flow}
Fix $\w\in\Omega$.  Assume a function  $B:\Z^2\times\Z^2\to\R$ satisfies 
	\[B(x,y)+B(y,z)=B(x,z)\quad\text{and}\quad\w_x=B(x,x+e_1)\wedge B(x,x+e_2) \quad\forall \, x,y,z\in\Z^2.\]

\smallskip

\begin{enumerate}[\ \ \rm(a)]
\item\label{lm:grad flow:a}  Let $x_{m,n}=(x_k)_{k=m}^n$ be any up-right path  that  follows  minimal  gradients of $B$, that is, 
\[\w_{x_k}=B(x_k,x_{k+1}) \qquad\text{for all $m\le k<n$.}\]
  Then $x_{m,n}$ is a geodesic from $x_m$ to $x_n$:
	\be\label{g:GB} \Gpp_{x_m,x_n}(\w)=\sum_{k=m}^{n-1}\w_{x_k}=B(x_m,x_n).\ee
\item\label{lm:grad flow:b}
Let $x_{m,n}=(x_k)_{k=m}^n$ be an up-right path such that for all $m\le k<n$  
\begin{align*}
\begin{split}
\text{either}\quad\w_{x_k}&=B(x_k,x_{k+1})<B(x_k,x_k+e_1)\vee B(x_k,x_k+e_2)\\
\text{or}\quad x_{k+1}&=x_k+e_2 \, \text{ and } \, B(x_k,x_k+e_1)=B(x_k,x_k+e_2).
\end{split}
\end{align*}
In other words, path $x_{m,n}$ follows   minimal gradients of $B$    and  takes an $e_2$-step in a tie. 
Then $x_{m,n}$ is the leftmost geodesic from $x_m$ to $x_n$.  Precisely,  if $x'_{m,n}$ is an up-right path from $x'_m=x_m$ to $x'_n=x_n$ and $\Gpp_{x_m,x_n}=\sum_{k=m}^{n-1}\w_{x'_k}$,
then $x_k\cdot e_1\le x'_k\cdot e_1$ for all $m\le k\le n$.

If ties are broken by   $e_1$-steps the resulting geodesic is   the rightmost geodesic between $x_m$ and $x_n$:
$x_k\cdot e_1\ge x'_k\cdot e_1$ for all $m\le k<n$.
\end{enumerate}
\end{lemma}

\begin{proof}
Part \eqref{lm:grad flow:a}. Any up-right path $y_{m,n}$ from $y_m=x_m$ to $y_n=x_n$ satisfies 
	\[\sum_{k=m}^{n-1}\w_{y_k}\le \sum_{k=m}^{n-1}B(y_k,y_{k+1})=B(x_m,x_n)=\sum_{k=m}^{n-1}B(x_k,x_{k+1})=\sum_{k=m}^{n-1}\w_{x_k}.  \]

Part \eqref{lm:grad flow:b}. 
$x_{m,n}$ is a geodesic by part \eqref{lm:grad flow:a}. To prove that it is the leftmost geodesic 
assume $x'_k=x_k$ and $x_{k+1}=x_k+e_1$. Then $\w_{x_k}=B(x_k,x_k+e_1)<B(x_k,x_k+e_2)$.   Recovery of the weights gives    $\Gpp_{x,y}\le B(x,y)$ for all $x\le y$. 
Combined   with \eqref{g:GB},  
	\[\w_{x_k}+\Gpp_{x_k+e_2,x_n}<B(x_k,x_k+e_2)+B(x_k+e_2,x_n)=B(x_k,x_n)=\Gpp_{x_k,x_n}.\]
Hence also  $x'_{k+1}=x'_k+e_1$ and  the claim about being the leftmost geodesic is proved. The other claim is symmetric. \qed
\end{proof}

Next we define cocycle geodesics, that is, geodesics constructed by following minimal gradients of   cocycles $B^{\xi}_{\pm}$  constructed in Theorem \ref{th:cocycles}.   Since our treatment allows discrete distributions, we introduce a    function $\t$ on $\Z^2$   to resolve ties.   
For $\xi\in\ri\Uset$, $u\in\Z^2$, and $\t\in\{e_1,e_2\}^{\Z^2}$, let $x^{u,\t,\xi,\pm}_{0,\infty}$ be the up-right path (one path for $+$, one for $-$) starting at $x^{u,\t,\xi,\pm}_0=u$ and satisfying for all $n\ge0$
	\[x^{u,\t,\xi,\pm}_{n+1}
	=\begin{cases}
	x^{u,\t,\xi,\pm}_n+e_1&\text{if }
	B^{\xi}_{\pm}(x^{u,\t,\xi,\pm}_n,x^{u,\t,\xi,\pm}_n+e_1)\\[2pt]
	       & \qquad\qquad  <B^{\xi}_{\pm}(x^{u,\t,\xi,\pm}_n,x^{u,\t,\xi,\pm}_n+e_2),\\[5pt] 
	x^{u,\t,\xi,\pm}_n+e_2&\text{if }
	B^{\xi}_{\pm}(x^{u,\t,\xi,\pm}_n,x^{u,\t,\xi,\pm}_n+e_2)\\[2pt]
	       & \qquad\qquad <B^{\xi}_{\pm}(x^{u,\t,\xi,\pm}_n,x^{u,\t,\xi,\pm}_n+e_1),\\[5pt]
	x^{u,\t,\xi,\pm}_n+\t(x^{u,\t,\xi,\pm}_n)&\text{if }
	B^{\xi}_{\pm}(x^{u,\t,\xi,\pm}_n,x^{u,\t,\xi,\pm}_n+e_1)\\[2pt]
	       & \qquad\qquad =B^{\xi}_{\pm}(x^{u,\t,\xi,\pm}_n,x^{u,\t,\xi,\pm}_n+e_2).
	\end{cases}\]
Cocycles 	
  $B^{\xi}_{\pm}$  satisfy  $\w_x=B^{\xi}_{\pm}(\what, x,x+e_1)\wedge B^{\xi}_{\pm}(\what, x,x+e_2)$   
  (Theorem \ref{th:cocycles}(\ref{th:cocycles:exist})) and so  by Lemma \ref{lm:grad flow}\eqref{lm:grad flow:a}, $x^{u,\t,\xi,\pm}_{0,\infty}$ is a semi-infinite geodesic.  
  
Through the cocycles  these geodesics are measurable functions on $\Ombig$.  
    If $\gpp$ is differentiable at the endpoints of its linear segments (if any),  cocycles   $B^{\xi}_{\pm}$ are  functions of $\w$ (Theorem~5.3 in \cite{Geo-Ras-Sep-15a-}).  Then   
 geodesics $x^{u,\t,\xi,\pm}_{0,\infty}$ can be defined on $\Omega$ without the artificial extension to the space $\Ombig=\Omega\times\Omega'$.

If we restrict ourselves to the event $\Ombig_0$ of full $\Pbig$-measure on which monotonicity 
  \eqref{eq:monotone} holds for all $\xi,\zeta\in\ri\Uset$,  we can order  these geodesics in a  natural way from left to right.  
Define a partial ordering 
on $\{e_1,e_2\}^{\Z^2}$ by   $e_2\preceq e_1$   
and then $\t\preceq\t'$   coordinatewise.   Then  on the event $\Ombig_0$,  
for any  $u\in\Z^2$, $\t\preceq\t'$,  
$\xi,\zeta\in\ri\Uset$ with $\xi\cdot e_1<\zeta\cdot e_1$, and for all $n\ge0$,
\be \label{eq:geod-order}	\begin{aligned}
	 x_n^{u,\t,\xi,\pm}\cdot e_1\le x_n^{u,\t',\xi,\pm}\cdot e_1,\ \ 
	 &x_n^{u,\t,\xi,-}\cdot e_1\le x_n^{u,\t,\xi,+}\cdot e_1, \\
	 \text{and }\ 
	 &x_n^{u,\t,\xi,+}\cdot e_1\le x_n^{u,\t,\zeta,-}\cdot e_1.
	\end{aligned} \ee
	
  The  leftmost and rightmost tie-breaking rules are   $\tmin(x)=e_2$ and $\tmax(x)=e_1$  $\forall x\in\Z^2$. 
 The cocycle limits \eqref{eq:cont} and \eqref{eq:cont-left} force the cocycle geodesics to converge also, as the next lemma shows. 

\begin{lemma}\label{lm:cont geod}
Fix $\xi$ and let  $\zeta_n\to\xi$ in $\ri\Uset$.  
If  
$\zeta_n\cdot e_1>\xi\cdot e_1$  $\forall n$  then for all $u\in\Z^2$ 
	\begin{align}\label{eq:cont geod}
	\Pbig\{\forall k\ge0\;\exists n_0<\infty:n\ge n_0\,\Rightarrow \,  x_{0,k}^{u,\,\tmax,\,\zeta_n,\pm}=x_{0,k}^{u,\,\tmax,\,\xi,+}\}=1.
	\end{align}
Similarly,  if $\zeta_n\cdot e_1\nearrow\xi\cdot e_1$,  we have  the almost sure   coordinatewise limit   $x_{0,\infty}^{u,\,\tmin,\,\zeta_n,\pm} \to x_{0,\infty}^{u,\,\tmin,\,\xi,-}$.    
\end{lemma}

\begin{proof}
It is enough to prove the statement for $u=0$. By  \eqref{eq:cont} and \eqref{eq:cont-left}, for a given $k$ and  large enough $n$,  if $x\ge0$ with $\abs{x}_1\le k$ and
$B^{\xi}_{+}(x,x+e_1)\ne B^{\xi}_{+}(x,x+e_2)$, then $B^{\zeta_n}_{\pm}(x,x+e_1)-B^{\zeta_n}_{\pm}(x,x+e_2)$ does not vanish and has the same sign as $B^{\xi}_{+}(x,x+e_1)-B^{\xi}_{+}(x,x+e_2)$.
From such  $x$ geodesics following the minimal gradient of $B^{\zeta_n}_{\pm}$ or the minimal gradient of $B^{\xi}_{+}$ stay together for their next step.
On the other hand, when $B^{\xi}_{+}(x,x+e_1)= B^{\xi}_{+}(x,x+e_2)$, monotonicity \eqref{eq:monotone} implies 
	\[B^{\zeta_n}_{\pm}(x,x+e_1)\le B^{\xi}_{+}(x,x+e_1)=B^{\xi}_{+}(x,x+e_2)\le B^{\zeta_n}_{\pm}(x,x+e_2).\]
Once again, both the geodesic following the minimal gradient of $B^{\zeta_n}_{\pm}$ and rules $\tmax$ and the one following the minimal gradients of $B^{\xi}_{+}$ and rules $\tmax$ will next take 
the same $e_1$-step. This proves \eqref{eq:cont geod}. The other claim is similar. \qed
\end{proof}

Recall the segments  $\Uset_\xi$,  $\Uset_{\xi\pm}$ in $\ri\Uset$ defined in \eqref{eq:sector1}--\eqref{eq:sector2} for $\xi\in\ri\Uset$.    
%
The next theorem concerns the direction of  cocycle  geodesics.

\begin{theorem}\label{th:B-geod:direction}
There exists an event $\Ombig_0$ such that $\Pbig(\Ombig_0)=1$ and for each $\what\in\Ombig_0$ the following holds: 
 	\be\label{g:77}  \forall\xi\in\ri\Uset, \, \forall\t\in\{e_1,e_2\}^{\Z^2}, \, \forall u\in\Z^2:x^{u,\t,\xi,\pm}_{0,\infty}\text{ is $\Uset_{\xi\pm}$-directed}.\ee
	For  $\xi\in\Diff$  the $\pm$ is immaterial in the statement.  \end{theorem}

\begin{proof}
Fix $\xi\in\ri\Uset$ and abbreviate $x_n=x^{u,\tmax,\xi,+}_n$.    
$\Gpp_{u,x_n}=B^{\xi}_{+}(u,x_n)$ by Lemma \ref{lm:grad flow}\eqref{lm:grad flow:a}.  
Apply Theorem \ref{th:Atilla} with cocycle $B^{\xi}_{+}$ to write
	\[\lim_{n\to\infty}\abs{x_n}_1^{-1} ( \Gpp_{u,x_n} - \nabla\gpp(\xi+)\cdot x_n  ) =0\qquad\Pbig\text{-almost surely}.\]

Define $\zeta(\what)\in\Uset$ by $\zeta\cdot e_1 = \varlimsup \frac{x_n\cdot e_1}{\abs{x_n}_1}$. 
If $\zeta\cdot e_1 >  \ximax\cdot e_1 $ then $\zeta\not\in\Uset_{\xi+}$ and hence
	\[
	\gpp(\zeta)  - \nabla \gpp(\xi+) \cdot\zeta  <  \gpp(\xi)  - \nabla\gpp(\xi+) \cdot\xi = 0.\]  
(The 
equality is from 
\eqref{euler}.
For the inequality, concavity gives $\le$ and $\zeta\not\in\Uset_{\xi+}$ rules out equality.)

Consequently, by the shape theorem   \eqref{lln5}, 
on the event $\{\zeta\cdot e_1>\ximax\cdot e_1\}$,   
\[  \varliminf_{n\to\infty}  \abs{x_n}_1^{-1} ( \Gpp_{u,x_n} -\nabla\gpp(\xi+)\cdot x_n ) <0.  \] 
This proves that 
\[\Pbig\Bigl\{\,\varlimsup_{n\to\infty}  \frac{x_n^{u,\tmax,\xi,+}\cdot e_1}{\abs{x_n^{u,\tmax,\xi,+}}_1}\le\ximax\cdot e_1\Bigr\}=1. \]
Repeat the same argument with $\tmax$ replaced by $\tmin$ and $\ximax$ by the other endpoint of 
$\Uset_{\xi+}$ (which is either $\xi$ or $\ximin$). To capture all $\t$ use   geodesics ordering \eqref{eq:geod-order}.   An analogous argument works for $\xi-$.  We have,  for a given $\xi$,  
	\be\label{g:78}  \Pbig\Bigl\{  \forall\t\in\{e_1,e_2\}^{\Z^2},\forall u\in\Z^2:x_{0,\infty}^{u,\t,\xi,\pm}\text{  is }\Uset_{\xi\pm}\text{-directed}\Big\}=1.\ee
 
Let $\Ombig_0$ be an event of full $\Pbig$-probability on which all cocycle  geodesics satisfy the ordering \eqref{eq:geod-order}, and   the event in \eqref{g:78} holds for both $+$ and $-$ and for
$\xi$ in  a countable set $\Uset_0$ that contains all points of nondifferentiability of $\gpp$ and a countable dense subset of $\Diff$.     We argue that   \eqref{g:77}  holds on $\Ombig_0$.  

Let $\zeta\notin\Uset_0$ and let $\zetamax$ denote the right endpoint of $\Uset_{\zeta}$.  We show that 
\be\label{g:79}   \varlimsup_{n\to\infty}  \frac{x_n^{u,\tmax,\zeta}\cdot e_1}{\abs{x_n^{u,\tmax,\zeta}}_1}\le\zetamax\cdot e_1 \qquad\text{on the event $\Ombig_0$.} \ee
(Since $\zeta\in\Diff$   there is no $\pm$ distinction in the  cocycle geodesic.)    The $\varliminf $ with $\tmin$ and $\ge \zetamin\cdot e_1$ comes of course with the same argument.  

If $\zeta\cdot e_1< \zetamax\cdot e_1$ pick  $\xi\in\Diff\cap\Uset_0$ so  that  $\zeta\cdot e_1< \xi\cdot e_1<\zetamax\cdot e_1$.    Then $\ximax=\zetamax$ and  \eqref{g:79} follows from the ordering.  

If $\zeta= \zetamax$, let $\e>0$ and  pick  $\xi\in\Diff\cap\Uset_0$ so  that  
$\zeta\cdot e_1< \xi\cdot e_1\le \ximax\cdot e_1<\zeta\cdot e_1+\e$.  This is possible because $\nabla\gpp(\xi)$ converges to but never equals $\nabla\gpp(\zeta)$  as  $\xi\cdot e_1\searrow \zeta\cdot e_1$.  Again by the ordering 
\[     \varlimsup_{n\to\infty}  \frac{x_n^{u,\tmax,\zeta}\cdot e_1}{\abs{x_n^{u,\tmax,\zeta}}_1}\le
\varlimsup_{n\to\infty}  \frac{x_n^{u,\tmax,\xi}\cdot e_1}{\abs{x_n^{u,\tmax,\xi}}_1}\le\ximax\cdot e_1 <\zeta\cdot e_1+\e. \] 
This completes the proof of Theorem \ref{th:B-geod:direction}.  \qed
 \end{proof}



\begin{lemma}\noindent\label{lm:sandwich}
\begin{enumerate}[\ \ \rm(a)]
\item \label{lm:sandwich:a}  Fix $\xi\in\ri\Uset$.  Then the  following  holds $\Pbig$-almost surely.  For any  semi-infinite geodesic $x_{0,\infty}$  
\begin{align}\label{eq:sandwich}
\varlimsup_{n\to\infty}\frac{x_n\cdot e_1}{\abs{x_n}_1} \ge \xi\cdot e_1 \quad\text{implies that}\quad 
		x_n\cdot e_1  \ge  x^{x_0,\,\tmin,\,\ximin,-}_n\cdot e_1 \quad\text{for all }n\ge0   
		\end{align}
and 
\begin{align}\label{eq:sandwich2}
\varliminf_{n\to\infty}\frac{x_n\cdot e_1}{\abs{x_n}_1} \le \xi\cdot e_1 \quad\text{implies that}\quad 
		x_n\cdot e_1  \le  x^{x_0,\,\tmax,\,\ximax,+}_n\cdot e_1 \quad\text{for all }n\ge0. 
		\end{align}
		
	
\item\label{lm:sandwich:b} Fix a maximal line segment $[\ximin, \ximax]$ on which $\gpp$ is linear and such that $\ximin\cdot e_1<\ximax\cdot e_1$.  Assume 	$\ximin, \ximax\in\Diff$.    Then the  following statement holds $\Pbig$-almost surely.  Any semi-infinite  geodesic $x_{0,\infty}$   such that   a limit point of  $x_n/\abs{x_n}_1$  lies in $[\ximin, \ximax]$
satisfies 
	\begin{align}\label{eq:sandwich4}
		x^{x_0,\tmin,\ximin}_n\cdot e_1\le x_n\cdot e_1\le x^{x_0,\tmax,\ximax}_n\cdot e_1\quad\text{for all }n\ge0.
	\end{align}
\end{enumerate}
\end{lemma}


\begin{proof}  Part (a).  
We prove    \eqref{eq:sandwich}.  \eqref{eq:sandwich2} is proved similarly. 
  Fix a sequence 
 $\zeta_\ell\in\Diff$ such that   $\zeta_\ell\cdot e_1\nearrow \ximin\cdot e_1$ so that, in particular,  $\xi\not\in\Uset_{\zeta_\ell}$.     The good event of full $\Pbig$-probability is the one on which 
$x_{0,\infty}^{x_0,\tmin,\zeta_\ell}$ is $\Uset_{\zeta_\ell}$-directed (Theorem \ref{th:B-geod:direction}),    $x_{0,\infty}^{x_0,\tmin,\zeta_\ell}$ is the leftmost geodesic between any two of its points   (Lemma \ref{lm:grad flow}\eqref{lm:grad flow:b}  applied to cocycle $B^{\zeta_\ell}$) and  $x_{0,\infty}^{x_0,\tmin,\zeta_\ell}\to x^{x_0,\tmin,\ximin,-}_{0,\infty}$  (Lemma \ref{lm:cont geod}).  

By the leftmost property,  if 
  $x_{0,\infty}^{x_0,\tmin,\zeta_\ell}$ ever goes strictly to the right of
$x_{0,\infty}$, these two geodesics  cannot touch   again at any later time.  
But   by virtue of the limit points,  $x_n^{x_0,\tmin,\zeta_\ell}\cdot e_1<x_n\cdot e_1$ for  infinitely many $n$.  
Hence  $x_{0,\infty}^{x_0,\tmin,\zeta_\ell}$ stays  weakly to the left of $x_{0,\infty}$.  Let $\ell\to\infty$.  
 
Part (b) is proved similarly.  The differentiability assumption implies that the geodesic $x_{0,\infty}^{x_0,\tmin,\ximin}$ can be approached from the left by geodesics $x_{0,\infty}^{x_0,\tmin,\zeta_\ell}$ such that  $\ximin\not\in\Uset_{\zeta_\ell}$.  \qed
\end{proof}

The next result concerns coalescence of cocycle geodesics $\{x^{u,\t,\xi,\pm}_{0,\infty}:u\in\Z^2\}$ for fixed $\t$, $\pm$,  and $\xi\in\ri\Uset$.    We can consider a random, stationary tie-breaking function $\t:\Ombig\times\Z^2\to\{e_1,e_2\}$ that satisfies 
\be\label{stat-t}   \t(\what, x)=\t(T_x\what, 0) \qquad\text{$\forall x\in\Z^2$, $\Pbig$-almost surely.}  \ee 

\begin{theorem}\label{th:geod:coal:t}
  Fix a tie-breaking function $\t$ that satisfies \eqref{stat-t} and fix $\xi\in\ri\Uset$. 
Then $\Pbig$-almost surely, for all $u,v\in\Z^2$, there exist $n,m\ge0$ such that
$x^{u,\t,\xi,-}_{n,\infty}=x^{v,\t,\xi,-}_{m,\infty}$, with a similar statement for $+$. 
\end{theorem}

Theorem \ref{th:geod:coal:t} is proved by  adapting the argument  of \cite{Lic-New-96},  originally presented  for first passage percolation, and by utilizing the independence property  in   Theorem \ref{th:cocycles}\eqref{th:cocycles:indep}.   
  Briefly, the idea is the following.      By stationarity the assumption of two nonintersecting geodesics implies we can find at least three.  A   local modification of the weights  turns  the middle geodesic of the triple   into a geodesic that stays disjoint from all geodesics that emanate from sufficiently far away.   By stationarity   at least $\delta L^2$ such disjoint geodesics   emanate from an $L\times L$ square.  This gives a contradiction because there are only $2L$ boundary points for these geodesics to exit through.  
Details can be found in Appendix A of the arXiv preprint \cite{Geo-Ras-Sep-15-} of this paper. 
  
\smallskip

The   coalescence result above rules out  the existence of doubly infinite cocycle geodesics (a.s.\ for a given cocycle). The following theorem gives the rigorous statement. Its proof is given at the end of the section and is
based again on a lack-of-space argument, similar to the proof of Theorem 6.9 in \cite{Dam-Han-14}. 

\begin{theorem}\label{th:d-geod:t}
  Let $\t$ be a stationary tie-breaker as in \eqref{stat-t}  and $\xi\in\ri\Uset$. 
Then $\Pbig$-almost surely, for all $u\in\Z^2$, there exist at most finitely many $v\in\Z^2$ such that 
$x^{v,\t,\xi,-}_{0,\infty}$ goes through $u$. The same statement holds for $+$. 
\end{theorem}

\smallskip 
 
It is known that, in general, uniqueness of geodesics cannot hold simultaneously for all directions.  In our development this is a consequence of Theorem \ref{th:two-geo} below.  
As a step towards uniqueness of geodesics in a given  direction,   the next lemma shows that continuity of the distribution of   $\w_0$ prevents   ties in \eqref{Bw-9}.    
(The construction of the cocycles implies, through equation (7.6) 
in \cite{Geo-Ras-Sep-15a-},   
that  variables $B^{\xi}_{\pm}(x,x+e_i)$ have continuous marginal distributions. Here we need a property of the joint distribution.)    
Consequently,   for a given $\xi$,   $\Pbig$-almost surely    geodesics $x_{0,\infty}^{u,\t,\xi,\pm}$ do not depend on $\t$.

\begin{lemma}\label{lm:geod:unique:t}
Assume  
 that $\P\{\w_0\le r\}$ is a continuous function of $r\in\R$. Fix $\xi\in\ri\Uset$. Then for all $u\in\Z^2$, 
\[\Pbig\{B^{\xi}_{+}(u,u+e_1)=B^{\xi}_{+}(u,u+e_2)\}=\Pbig\{B^{\xi}_{-}(u,u+e_1)=B^{\xi}_{-}(u,u+e_2)\}=0.\]
\end{lemma}


\begin{proof} 
Due to shift invariance it is enough to prove the claim for $u=0$. 
We  work with the $+$ case, the other case being similar.

Assume by way of contradiction that the probability in question is positive. 
By Theorem \ref{th:geod:coal:t}, $x_{0,\infty}^{e_2,\tmax,\xi,+}$ and $x_{0,\infty}^{e_1,\tmax,\xi,+}$   coalesce with probability one. 
Hence there exists   $v\in\Z^2$ and  $n\ge1$ such that
	\[\P\big\{B^{\xi}_{+}(0,e_1)=B^{\xi}_{+}(0,e_2),\,x_n^{e_1,\tmax,\xi,+}=x_n^{e_2,\tmax,\xi,+}=v\big\}>0.\]
Note that if $B^{\xi}_{+}(0,e_1)=B^{\xi}_{+}(0,e_2)$ then both are equal to $\w_0$. Furthermore, by Lemma \ref{lm:grad flow}\eqref{lm:grad flow:a} we have
	\[B^{\xi}_{+}(e_1,v)=\sum_{k=0}^{n-1}\w(x_k^{e_1,\tmax,\xi,+})\quad\text{and}\quad B^{\xi}_{+}(e_2,v)=\sum_{k=0}^{n-1}\w(x_k^{e_2,\tmax,\xi,+}).\]
(For aesthetic reasons we wrote $\w(x)$ instead of $\w_x$.) Thus
	\begin{align*}
	&\w_0+\sum_{k=0}^{n-1}\w({x_k^{e_1,\tmax,\xi,+}})=B^{\xi}_{+}(0,e_1)+B^{\xi}_{+}(e_1,v)= B^{\xi}_{+}(0,v)\\
	&\qquad=B^{\xi}_{+}(0,e_2)+B^{\xi}_{+}(e_2,v)=\w_0+\sum_{k=0}^{n-1}\w({x_k^{e_2,\tmax,\xi,+}}).
	\end{align*}
The fact that this happens with positive probability contradicts the assumption that $\w_x$ are i.i.d.\ and have a continuous distribution. \qed 
\end{proof}

\begin{proofof}{of Theorem \ref{thm1}}
Part \eqref{thm1:exist}.   The existence of $\Uset_{\xi\pm}$-directed semi-infinite geodesics for $\xi\in\ri\Uset$  follows by fixing $\t$ and  taking geodesics  $x^{u,\t,\xi,\pm}_{0,\infty}$ from  Theorem \ref{th:B-geod:direction}.   For $\xi=e_i$  semi-infinite geodesics are simply  $x_{0,\infty}=(x_0+ne_i)_{n\ge0}$.   

 Let  $\Diff_0$ be a dense countable subset of $\Diff$.  Let   $\Ombig_0$  be the event   of full $\Pbig$-probability on which event \eqref{g:77} holds  
and Lemma \ref{lm:sandwich}\eqref{lm:sandwich:a} holds  for each  $u\in\Z^2$ and   $\zeta\in\Diff_0$.  We  show that   on  $\Ombig_0$,  every semi-infinite geodesic is $\Uset_\xi$-directed for some $\xi\in\Uset$.  

  Fix $\what\in\Ombig_0$ and an arbitrary semi-infinite geodesic $x_{0,\infty}$.  
Define $\xi'\in\Uset$ by 
\[   \xi'\cdot e_1=\varlimsup_{n\to\infty}\frac{x_n\cdot e_1}{\abs{x_n}_1}.  \]  
Let $\xi=\ximin'=$ the left endpoint of $\Uset_{\xi'}$.  
We claim that   $x_{0,\infty}$ is $\Uset_\xi=[\ximin, \ximax]$-directed.  
If $\xi'=e_2$ then  $x_n/\abs{x_n}_1\to e_2$ and $\Uset_\xi=\{e_2\}$ and the case is closed.
Suppose  $\xi'\ne e_2$.

The definition of $\xi$ implies that $\xi'\in\Uset_{\xi+}$ and so 
\[  \varlimsup_{n\to\infty}\frac{x_n\cdot e_1}{\abs{x_n}_1} =\xi'\cdot e_1\le     \ximax\cdot e_1.  \]

From the other direction,  for any $\zeta\in\Diff_0$ such that $ \zeta\cdot e_1<  \xi'\cdot e_1$ we have 
\[  \varlimsup_{n\to\infty}\frac{x_n\cdot e_1}{\abs{x_n}_1} > \zeta\cdot e_1 \]
which by \eqref{eq:sandwich} implies  $x_n\cdot e_1  \ge  x^{x_0,\tmin,\zetamin}_n\cdot e_1$. 
Then  by  \eqref{g:77} 
\[  \varliminf_{n\to\infty}\frac{x_n\cdot e_1}{\abs{x_n}_1} \; \ge \; \varliminf_{n\to\infty} \frac{x^{x_0,\tmin,\zetamin}_n\cdot e_1}{\abs{x^{x_0,\tmin,\zetamin}_n}_1} 
\;\ge\; \underline\zetamin\cdot e_1 \]
where  $\underline\zetamin=$  the left endpoint of $\Uset_{\zetamin}$.  It remains to observe that we can take $\underline\zetamin\cdot e_1$ arbitrarily close to $\ximin\cdot e_1$.    If $ \xi\cdot e_1<  \xi'\cdot e_1$  then we   take  $ \xi\cdot e_1< \zeta\cdot e_1<  \xi'\cdot e_1$ in which case $\zetamin=\xi$ and $\underline\zetamin=\ximin$.    If $\xi=\xi'$ then also $\ximin=\ximin'=\xi$. In this case,  as  $\Diff_0\ni\zeta\nearrow \xi$, $\nabla\gpp(\zeta)$ approaches but never equals $\nabla\gpp(\xi-)$ because there is no flat segment of $\gpp$ adjacent to $\xi$ on the left.  This forces both $\zetamin$ and $\underline\zetamin$ to converge to $\xi$.


\smallskip 

Part  \eqref{thm1:direct}.   If $\gpp$ is strictly concave then  $\Uset_\xi=\{\xi\}$ for all $\xi\in\ri\Uset$ and \eqref{thm1:exist} $\Rightarrow$ \eqref{thm1:direct}.  

\smallskip 

Part \eqref{thm1:cont}.   
By Theorem \ref{th:cocycles}\eqref{th:cocycles:flat} there is a single cocycle $B^\zeta$ simultaneously  for all $\zeta\in[\ximin,\ximax]$.   Consequently cocycle geodesics 
$x_{0,\infty}^{x_0,\t,\ximin}$ and $x_{0,\infty}^{x_0,\t,\ximax}$ coincide for any given tie breaking function $\t$.  
On the event of full $\P$-probability   on which there are no ties between $B^\zeta(x,x+e_1)$ and $B^\zeta(x,x+e_2)$ the tie breaking function $\t$ makes no difference.  Hence 
the left and right-hand side of \eqref{eq:sandwich4} coincide.   Thus there is no room for two $[\ximin,\ximax]$-directed geodesics from any point.    Coalescence comes from Theorem \ref{th:geod:coal:t}.
 The statement about the finite number of ancestors of a site $u$ comes from Theorem \ref{th:d-geod:t}.
 \qed\end{proofof}

\begin{proofof}{of Theorem \ref{thm:buse-geo}}
By Theorem \ref{th:construction} limit $B$ from \eqref{eq:grad:coc1}  is now the cocycle $B^\xi$.
Part \eqref{thm:buse-geo:i} follows from Lemma \ref{lm:grad flow}.\medskip

Part \eqref{thm:buse-geo:ii}. 
Take sequences $\eta_n, \zeta_n\in\ri\Uset$ with $\eta_n\cdot e_1<\ximin\cdot e_1\le \ximax\cdot e_1<\zeta_n\cdot e_1$ and $\zeta_n\to\ximax$,  $\eta_n\to\ximin$. Consider  the full measure event on which Theorem \ref{th:construction} holds for each  $\zeta_n$ and $\eta_n$ 
with sequences $v_m=\fl{m\zeta_n}$ and $\fl{m\eta_n}$, and on  which 
continuity \eqref{eq:cont} and \eqref{eq:cont-left} holds as $\zeta_n\to\ximax$,  $\eta_n\to\ximin$. In the rest of the proof we drop the index $n$ from $\zeta_n$ and $\eta_n$.

We prove the case of a semi-infinite geodesic 
  $x_{0,\infty}$   that satisfies $x_0=0$ and  \eqref{geod-98}. 
For large $m$, $\fl{m\eta\cdot e_1}<x_m\cdot e_1<\fl{m\zeta\cdot e_1}$.

Consider first the case $x_1=e_1$. If there exists a geodesic from $0$ to $\fl{m\zeta}$ that goes through $e_2$, then this geodesic would intersect $x_{0,\infty}$ and thus there would
exist another geodesic that goes from $0$ to $\fl{m\zeta}$ passing through $e_1$. In this case we would have $\Gpp_{e_1,\fl{m\zeta}}=\Gpp_{e_2,\fl{m\zeta}}$. On the other hand, if there exists a geodesic from $0$ to $\fl{m\zeta}$ that goes through $e_1$, then we would have $\Gpp_{e_1,\fl{m\zeta}}\ge\Gpp_{e_2,\fl{m\zeta}}$. Thus,
in either case, we have 
	\[\Gpp_{0,\fl{m\zeta}}-\Gpp_{e_1,\fl{m\zeta}}\le \Gpp_{0,\fl{m\zeta}}-\Gpp_{e_2,\fl{m\zeta}}.\]
Taking $m\to\infty$ and applying Theorem \ref{th:construction} we have $B^{\overline\zeta}_{+}(0,e_1)\le B^{\overline\zeta}_{+}(0,e_2)$.
Taking $\zeta\to\ximax$ and applying \eqref{eq:cont} we have $B^{\ximax}_{+}(0,e_1)\le B^{\ximax}_{+}(0,e_2)$.
Since $\ximax$ and $\xi$ are points of differentiability of $\gpp$, we have $B^{\ximax}_{+}=B^{\xi}$. Consequently, we have shown $B^\xi(0,e_1)\le B^{\xi}(0,e_2)$.  Since $B^\xi$ recovers the weights
\eqref{Bw-9},    the first step $x_1=e_1$ of $x_{0,\infty}$  satisfies $\w_{0}=B^\xi(0,e_1)\wedge B^{\xi}(0,e_2)=B^\xi(0, x_1)$.

When $x_1=e_2$ repeat the same argument   with $\eta$ in place of $\zeta$ to get  $B^{\xi}(0,e_2)\le B^\xi(0,e_1)$. This proves the theorem for the first step of the geodesic and that 
  is enough.\medskip

Part \eqref{thm:buse-geo:iii}. 
We prove the case $i=0$.
The statement holds if $B^\xi(0,e_1)=B^\xi(0,e_2)$, since then both are equal to $\w_0$ by  recovery \eqref{eq:VB}.
If $\w_0=B^\xi(0,e_1)<B^\xi(0,e_2)$ then convergence \eqref{eq:grad:coc1} implies that for $n$ large enough $\Gpp_{e_1,v_n}>\Gpp_{e_2,v_n}$.
In this case any maximizing path from $0$ to $v_n$ will have to start with an $e_1$-step and the claim is again true.
The case $B^\xi(0,e_1)>B^\xi(0,e_2)$ is similar.
\qed\end{proofof}

\begin{proofof}{of Theorem \ref{thm:lr-geod}}
Part \eqref{thm:lr-geod:i}.   $\underline\ximin=\ximin$ implies  $\Uset_{\ximin-}\subset\Uset_{\xi-}$
and, by Theorem \ref{th:B-geod:direction}, $x_{0,\infty}^{u,\,\tmin,\,\underline\ximin,-}=x_{0,\infty}^{u,\,\tmin,\,\ximin,-}$ 
is $\Uset_{\xi-}$-directed.
Lemma \ref{lm:sandwich}\eqref{lm:sandwich:a} implies that any $\Uset_{\xi-}$-directed semi-infinite geodesic out of $u\in\Z^2$ stays to the right of $x_{0,\infty}^{u,\,\tmin,\,\underline\ximin,-}$.   Thus  
 $x_{0,\infty}^{u,\,\tmin,\,\ximin,-}$ is the leftmost $\Uset_{\xi-}$-directed  geodesic out of $u$. 
The coalescence claim follows now from Theorem \ref{th:geod:coal:t} and 
the statement about the finite number of ancestors of a site $u$  comes from Theorem \ref{th:d-geod:t}. The case $\overline\ximax=\ximax$ is similar. Part \eqref{thm:lr-geod:i} is proved.

Part \eqref{thm:lr-geod:ii}. It is enough to work with the case $u=0$. The differentiability assumption implies $B^\xi_\pm=B^\xi$.  
 We will thus omit the $\pm$ from the $B^\xi$-geodesics notation.
Take $v_n$ as in \eqref{eq:vn}. Consider an up-right path $y_{0,\infty}$ that is a limit point of the sequence of leftmost geodesics from $0$ to $v_n$. By this we mean that along this subsequence, 
for any $m\in\N$ the initial $m$-step segment of the leftmost geodesic from $0$ to $v_n$
equals $y_{0,m}$ for $n$ large enough. By Theorem \ref{thm:buse-geo}\eqref{thm:buse-geo:iii} we have almost surely $B^\xi(y_i,y_{i+1})=\w_{y_i}$ for all $i\ge0$. Furthermore, for any $n\in\N$, 
$y_{0,n}$ is the leftmost geodesic between $0$ and $y_n$.
We will next show that whenever $B^\xi(y_i,y_i+e_1)=B^\xi(y_i,y_i+e_2)$ we have $y_{i+1}=e_2$. This then implies that $y_{0,\infty}=x_{0,\infty}^{0,\tmin,\xi}$ and proves part \eqref{thm:lr-geod:ii}.

It is enough to discuss the case of a tie at $y_0=0$. Assume that $B^\xi(0,e_1)=B^\xi(0,e_2)$ but $y_1=e_1$. 
By Theorem \ref{th:geod:coal:t}, $x_{0,\infty}^{e_2,\tmax,\xi}$ coalesces with $x_{0,\infty}^{e_1,\tmax,\xi}$. On the other hand, since we already know that $y_{1,\infty}$ follows minimal $B^\xi$-gradients we know that 
it must remain to the left of $x_{0,\infty}^{e_1,\tmax,\xi}$. This shows that $x_{0,\infty}^{e_2,\tmax,\xi}$ intersects $y_{1,\infty}$ at some point $z$ on level $n=\abs{z}_1$.
But now the path $\bar y_{0,n}$ with $\bar y_0=0$ and $\bar y_{1,n}=x_{0,n-1}^{e_2,\tmax,\xi}$ has last passage weights 	
	\[\w_0+\Gpp_{e_2,z}=\w_0+B^\xi(e_2,z)=B^\xi(0,e_2)+B^\xi(e_2,z)=B^\xi(0,z)\ge\Gpp_{0,z}\]
and is hence a geodesic. (The first equality is because $x_{0,n-1}^{e_2,\tmax,\xi}$ is a cocycle geodesic, the second  comes from weights recovery \eqref{eq:VB} and the tie at $0$, the third is additivity of cocycles,
and the last equality is again weights recovery \eqref{eq:VB}.) 
This contradicts the fact that $y_{0,n}$ is the leftmost geodesic from $0$ to $y_n=z$. 

We have thus shown that $y_{0,\infty}=x_{0,\infty}^{0,\tmin,\xi}$. A similar statement works for the rightmost geodesics.
Part \eqref{thm:lr-geod:ii} is proved.
\qed\end{proofof}
 
 \begin{proofof}{of Theorem \ref{th:d-geod:t}}
Let 		$C_u(\what)=\{v\in\Z^2:x^{v,\t,\xi,-}_{0,\infty}\text{ goes through }u\}.$  
The goal is $\Pbig\{\abs{C_u}=\infty\}=0$. Assume the contrary.  
 Since  $C_u$ is determined by the ergodic processes \eqref{erg-B},  there is then a positive density of  points $u\in\Z^2$ with $\abs{C_u}=\infty$.


Consider the tree $\cG$ made out of the union of geodesics $x_{0,\infty}^{x,\t,\xi,-}$ for $x\in\Z^2$. 
(The graph is a tree because once geodesics intersect they merge. It is connected due to coalescence  given by Theorem \ref{th:geod:coal:t}.)
Given $u_1,u_2\in\Z^2$ with $\abs{C_{u_1}}=\abs{C_{u_2}}=\infty$ 
consider the point where $x^{u_1,\t,\xi,-}_{0,\infty}$ and $x^{u_2,\t,\xi,-}_{0,\infty}$ coalesce. 
Removing this point from the tree splits the tree  into three infinite components. 
Call such a point a {\sl junction point}. 
At each junction point $u$    two infinite admissible paths meet for the first time  at $u$, and each path, as a cocycle geodesic, follows the minimal gradients of $B^{\xi}_{-}$ and 
uses tie-breaking rule $\t$.
We  call these the backward geodesics associated to $u$.

By shift-invariance and the argument above we have for all $u\in\Z^2$
	\begin{align}\label{junction}
	\Pbig\{u\text{ is a junction point}\}=\Pbig\{0\text{ is a junction point}\}>0.
	\end{align}
Then the ergodic theorem  implies that there is a positive density of junction points in $\Z^2$.
We   give a lack of space argument that   contradicts this.  

Let  $J$ be the set  of junction points in the box $[1,L]^2$ together with those points on the south and west boundaries $\{ke_i:1\le k\le L,i=1,2\}$ where a backward geodesic from a junction point first hits the boundary.  
Decompose  $J$  into finite binary trees by declaring that the two immediate descendants of a junction point are the two closest points  on its two backward geodesics that are members of $J$.  
Then the leaves of these trees are exactly the points on the boundary and the junction points are interior points of the trees.  A finite binary tree has more leaves than interior points.  
Consequently, there cannot  be more than $2L+1$ junction points inside $[1,L]^2$. This contradicts \eqref{junction} and proves the   theorem.
%
%
%
\qed\end{proofof}

\section{Competition interface}\label{sec:ci-pf} 

This section proves the results of Section \ref{sec:ci}.  As before, we begin by studying the situation on the extended space $\Ombig$ with the help of the cocycles $B^{\zeta}_{\pm}$ of Theorem \ref{th:cocycles}.

\begin{lemma}\label{aux:ci}
Define $B^{e_1}_{-}$ and $B^{e_2}_{+}$ as the monotone limits of $B^{\zeta}_{\pm}$ when $\zeta\to e_i$, $i=1,2$ respectively. 
Then $\Pbig$-almost surely $B^{e_1}_{-}(0,e_1)=B^{e_2}_+(0,e_2)=\w_0$ and $B^{e_1}_{-}(0,e_2)=B^{e_2}_+(0,e_1)=\infty$.
\end{lemma}

\begin{proof}
The limits  exist due to monotonicity \eqref{eq:monotone}. By   \eqref{Bw-9}  
$B^{e_1}_{-}(0,e_1)\ge\w_0$ almost surely. Dominated convergence and \eqref{eq:h=grad} give the limit 
	\[\Ebig[B^{e_1}_{-}(0,e_1)]=\lim_{\zeta\to e_1}\Ebig[B^{\zeta}_{\pm}(0,e_1)]=\lim_{\zeta\to e_1} e_1\cdot\nabla\gpp(\zeta\pm)=\Ebig[\w_0].\]
	The last equality is a consequence of \eqref{eq:g-asym} (see Lemma~4.1 and equations (4.12)--(4.13) in \cite{Geo-Ras-Sep-15a-}).  
  Now  $B^{e_1}_{-}(0,e_1)=\w_0$ almost surely. 

Additivity (Definition \ref{def:cK}\eqref{def:cK:coc}) and recovery \eqref{Bw-9}  are satisfied by $B^{e_1}_{-}$ and  imply
\begin{align*}
&B^{e_1}_{-}(ne_1,ne_1+e_2) \\
&\qquad =\w_{ne_1}+\big(B^{e_1}_{-}((n+1)e_1,(n+1)e_1+e_2)-B^{e_1}_{-}(ne_1+e_2,(n+1)e_1+e_2)\big)^+\\
&\qquad =\w_{ne_1}+\big(B^{e_1}_{-}((n+1)e_1,(n+1)e_1+e_2)-\w_{ne_1+e_2}\big)^+\,.
\end{align*}
The second equality is from the just proved identity $B^{e_1}_{-}(x,x+e_1)=\w_x$.

Repeatedly dropping the outer $+$-part and applying the same formula inductively leads to 
 \begin{align*} B^{e_1}_{-}(0,e_2)&\ge\w_0+\sum_{i=1}^{n}(\w_{ie_1}-\w_{(i-1)e_1+e_2})\\
 &\qquad + \big(B^{e_1}_{-}((n+1)e_1,(n+1)e_1+e_2)-\w_{ne_1+e_2}\big)^+.\end{align*} 
Since the  summands are i.i.d.\ with mean $0$, taking $n\to\infty$ gives  $B^{e_1}_{-}(0,e_2)=\infty$ almost surely.\qed
\end{proof}

\begin{lemma}\label{geo-sep}
Fix $\xi\in(\ri\Uset)\smallsetminus\Diff$ and a   tie-breaker $\t$ that satisfies \eqref{stat-t}. 
With $\Pbig$-probability one, for any $u\in\Z^2$ geodesics $x^{u,\t,\xi,+}_{0,\infty}$ and $x^{u,\t,\xi,-}_{0,\infty}$ eventually separate. 
\end{lemma}

\begin{proof}
Let $A_u=\{x^{u,\t,\xi,+}_{0,\infty}=x^{u,\t,\xi,-}_{0,\infty}\}$. We want  $\Pbig(A_0)=0$.
Assume the contrary.  Fix $\zeta\in\ri\Uset$. 
$\Pbig(A_0)>0$ and stationarity imply that with positive probability there exists a random sequence $u_n=\fl{k_n\zeta}$ such that  $k_n\to\infty$ and 
$A_{u_n}$ holds for each $n$. 
Furthermore, for each such $u_n$ we know  from Theorem \ref{th:geod:coal:t} that $x^{u_0,\t,\xi,+}_{0,\infty}=x^{u_0,\t,\xi,-}_{0,\infty}$ and $x^{u_n,\t,\xi,+}_{0,\infty}=x^{u_n,\t,\xi,-}_{0,\infty}$ coalesce 
at some random point $z_n$. By the additivity and \eqref{g:GB} we then have
	\begin{align}\label{previous}
	\begin{split}
	B^\xi_+(u_0,u_n)  &=   B^\xi_+(u_0,z_n) - B^\xi_+(u_n, z_n) = \Gpp_{u_0,z_n}-\Gpp_{u_n,z_n}\\
	& = B^\xi_-(u_0,z_n) - B^\xi_-(u_n, z_n) = B^\xi_-(u_0,u_n).
	\end{split}
	\end{align}
By recovery \eqref{Bw-9} the conditions of Theorem \ref{th:Atilla} are satisfied and because of \eqref{eq:h=grad} we have
	\begin{align*}
	&\lim_{n\to\infty}\frac{B^\xi_+(u_0,u_n)-\nabla\gpp(\xi+)\cdot(u_n-u_0)}{\abs{u_n}_1}\\
	&\qquad\qquad=0=\lim_{n\to\infty}\frac{B^\xi_-(u_0,u_n)-\nabla\gpp(\xi-)\cdot(u_n-u_0)}{\abs{u_n}_1}.
	\end{align*}
This and \eqref{previous} lead to $\nabla\gpp(\xi-)\cdot\zeta=\nabla\gpp(\xi+)\cdot\zeta$. Since $\zeta$ is arbitrary we get $\nabla\gpp(\xi-)=\nabla\gpp(\xi+)$,  
which contradicts the assumption on $\xi$.\qed
%
%
%
\end{proof}

Now  assume that $\w_0$ has a continuous distribution. 
By Lemma \ref{lm:geod:unique:t} we can omit $\t$ from the cocycle geodesics notation and write $x^{u,\xi,\pm}_{0,\infty}$.

Next we use the cocycles to define a random variable $\cid$ on $\Ombig$  that represents the asymptotic direction of the competition interface.    
By Lemma \ref{lm:geod:unique:t}, with $\Pbig$-probability one,  $B^{\xi}_{\pm}(0,e_1)\ne B^{\xi}_{\pm}(0,e_2)$ for all rational $\xi\in\ri\Uset$. Furthermore, monotonicity \eqref{eq:monotone}
gives that
	\[B^{\zeta}_{+}(0,e_1)-B^{\zeta}_{+}(0,e_2)\le B^{\zeta}_{-}(0,e_1)-B^{\zeta}_{-}(0,e_2) \le B^{\eta}_{+}(0,e_1)-B^{\eta}_{+}(0,e_2)\]
when $\zeta\cdot e_1>\eta\cdot e_1$. Lemma \ref{aux:ci} implies that $B^{\zeta}_{\pm}(0,e_1)-B^{\zeta}_{\pm}(0,e_2)$ converges to $-\infty$ as $\zeta\to e_1$ and to $\infty$ as $\zeta\to e_2$. 
Thus there exists   unique $\cid(\what)\in\ri\Uset$ such that   for rational $\zeta\in\ri\Uset$, 
\begin{align}\label{ci}
\begin{split}
&B^{\zeta}_{\pm}(\what, 0,e_1)<B^{\zeta}_{\pm}(\what, 0,e_2)\quad\text{if }\zeta\cdot e_1>\cid(\what)\cdot e_1\\
\text{ and}\qquad &B^{\zeta}_{\pm}(\what, 0,e_1)>B^{\zeta}_{\pm}(\what, 0,e_2)\quad\text{if }\zeta\cdot e_1<\cid(\what)\cdot e_1.
\end{split}
\end{align}
For the next theorem on the properties of $\cid$, recall  $\Uset_{\cid(\what)}=[\cidmin(\what),\cidmax(\what)]$ from \eqref{eq:sector2}.

\begin{theorem}\label{th:two-geo}
Assume $\P\{\w_0\le r\}$ is continuous in $r$. 
Then on the extended space $\OBPbig$ of Theorem \ref{th:cocycles}  the  random variable $\cid(\what)\in\ri\Uset$  defined by \eqref{ci}   has the following properties. 
\begin{enumerate}[\ \ \rm(i)]
\item\label{th:two-geo:i}  $\Pbig$-almost surely, for every $z\in\Z^2$, there exists a $\Uset_{\cid(T_z\what)-}$-directed geodesic out of $z$ that goes through $z+e_2$ and a
$\Uset_{\cid(T_z\what)+}$-directed geodesic out of $z$ that goes through $z+e_1$. The two geodesics intersect only  at their starting point $z$.
\item\label{th:two-geo:ii} The following holds $\Pbig$-almost surely. Let   $x'_{0,\infty}$ and $x''_{0,\infty}$ be any geodesics out of $0$ with   
\begin{align}\label{eq:thm5.3(ii)}
\varliminf_{n\to\infty}  \frac{x'_n\cdot e_1}n<\cidmin(\what) 
\qquad\text{and}\qquad 
\varlimsup_{n\to\infty}  \frac{x''_n\cdot e_1}n>\cidmax(\what).
\end{align}  
Then   $x'_1=e_2$ and  $x''_1=e_1$.  
\item\label{th:two-geo:iii} $\cid$ is almost surely an exposed point {\rm(}see \eqref{eq:epod} for the definition{\rm)}.  Furthermore,   $\Phat\{\what: \cid(\what)=\xi\}>0$ if and only if $\xi\in(\ri\Uset)\smallsetminus\Diff$.
\item\label{th:two-geo:iv}  Fix $u\in\Z^2$. 

{\rm(a)}  Let  $\zeta,\eta\in\ri\Uset$  be such that $\zeta\cdot e_1<\eta\cdot e_1$ and $\nabla\gpp(\zeta+)\not=\nabla\gpp(\eta-)$.
Then for $\Pbig$-almost every $\what$ there exists $z\in u+\Z_+^2$ such that $\cid(T_z\what)\in\;]\zeta,\eta[$.

{\rm(b)}  Let $\xi\in(\ri\Uset)\smallsetminus\Diff$.    Then for $\Pbig$-almost every $\what$ there exists $z\in u+\Z_+^2$ such that $\cid(T_z\what)=\xi$.

The point $z$ can be chosen so that,
in both cases {\rm(a)} and {\rm(b)},   there are two geodesics out of $u$ that split at this $z$ and after that never intersect, and of these two  geodesics the one  that goes through $z+e_2$ is 
$\Uset_{\cid(T_z\what)-}$-directed, while the one that goes through $z+e_1$ is 
$\Uset_{\cid(T_z\what)+}$-directed. 
\end{enumerate}
\end{theorem}


\begin{proof}
%
%
Fix a (possibly $\what$-dependent) $z\in\Z^2$. Define 
\begin{align}\label{B*} 
\begin{split}
\Bci_+(\what,x,y)&=\lim_{\eta\cdot e_1\searrow\cid(T_z\what)\cdot e_1}B^{\eta}_{\pm}(\what,x,y)  \\
\text{ and}\qquad \Bci_-(\what,x,y)&=\lim_{\zeta\cdot e_1\nearrow\cid(T_z\what)\cdot e_1}B^{\zeta}_{\pm}(\what,x,y)\,.
\end{split} \end{align}   
We have to keep the  $\Bci_{\pm}$ distinction  because the almost sure continuity statement  \eqref{eq:cont-left}  does not apply to the random 
direction $\cid$.  
In any case, $\Bci_{\pm}$ are additive (Definition \ref{def:cK}\eqref{def:cK:coc}) and recover weights $\w_x=\min_{i=1,2} \Bci_{\pm}(\what,x,x+e_i)$. 
From \eqref{ci} and stationarity (Definition \ref{def:cK}\eqref{def:cK:stat}) we have  
	\begin{align}\label{part ways}
	\Bci_+(z,z+e_1)\le \Bci_+(z,z+e_2)\quad\text{and}\quad\Bci_-(z,z+e_1)\ge \Bci_-(z,z+e_2).
	\end{align}

Fix any two tie-breaking rules $\t^+$ and $\t^-$ such that $\t^+(z)=e_1$ and $\t^-(z)=e_2$. 
By Lemma \ref{lm:grad flow} and \eqref{part ways} there exists  a geodesic from $z$   through $z+e_1$  (by following  minimal $\Bci_+$ gradients and using rule $\t^+$) and another through $z+e_2$ 
(by following  minimal $\Bci_-$ gradients and using rule $\t^-$). 
These two geodesics cannot coalesce because $\w_0$ has a continuous distribution.

Let $\zeta\cdot e_1<\cid(T_z\what)\cdot e_1.$ 
By  the limits in \eqref{B*} and monotonicity \eqref{eq:monotone}, 
\begin{align*}
\begin{split}
B^{\zeta}_{+}(\what,x,x+e_1)&\ge\Bci_-(\what,x,x+e_1)\ge B^{\cid(T_z\what)}_{-}(\what,x,x+e_1)\\
\text{ and}\qquad 
B^{\zeta}_{+}(\what,x,x+e_2)&\le\Bci_-(\what,x,x+e_2)\le B^{\cid(T_z\what)}_{-}(\what,x,x+e_2).
\end{split} \end{align*}  
These inequalities imply that the  
geodesics that follow the minimal gradients of $\Bci_-$ stay 
  to the right of $x_{0,\infty}^{z,\zeta,+}$   and to the left of $x_{0,\infty}^{z,\t^-,\cid(T_z\what),-}$. 
  By Theorem \ref{th:B-geod:direction} these latter geodesics are $\Uset_{\zeta+}$- and $\Uset_{\cid(T_z\what)-}$-directed, respectively.
Taking $\zeta\to\cid(T_z\what)$ shows the $\Bci_-$-geodesics are $\Uset_{\cid(T_z\what)-}$-directed. A similar argument gives that $\Bci_+$-geodesics are $\Uset_{\cid(T_z\what)+}$-directed.
Part \eqref{th:two-geo:i} is proved.

In part \eqref{th:two-geo:ii}  we prove the first claim, the other claim being similar.  The assumption allows us to pick a rational $\eta\in\ri\Uset$ such that
$\varliminf x_n'\cdot e_1/n<\underline\eta\cdot e_1\le\eta\cdot e_1<\cid\cdot e_1$. 
Since $\w_0$ has a continuous distribution and geodesic  $x_{0,\infty}^{0,\eta,-}$ is $\Uset_{\eta-}$-directed, geodesic $x'_{0,\infty}$ has to stay always to the left of it.
  \eqref{ci} implies $x_1^{0,\eta,-}=e_2$. Hence  also $x_1=e_2$. The claim is proved.

For part \eqref{th:two-geo:iii} fix first $\xi\in\Diff$, which implies   $B^{\xi}_{\pm}=B^\xi$.   By Lemma \ref{lm:geod:unique:t},   $B^\xi(0,e_1)\ne B^\xi(0,e_2)$ almost surely.    Let $\zeta\cdot e_1\searrow\xi\cdot e_1$ along rational points $\zeta\in\ri\Uset$.  By \eqref{eq:cont},  $B^{\zeta}_{\pm}(0, e_i)\to B^\xi(0,e_i)$ a.s.         Then on the event  $B^\xi(0,e_1)> B^\xi(0,e_2)$  there almost surely exists  a rational $\zeta$ such that $\zeta\cdot e_1>\xi\cdot e_1$ and  $B^{\zeta}_{\pm}(0,e_1)> B^{\zeta}_{\pm}(0,e_2)$.  By \eqref{ci} this forces  
$\cid\cdot e_1\ge\zeta\cdot e_1>\xi\cdot e_1$.   Similarly on the event $B^\xi(0,e_1)< B^\xi(0,e_2)$ we   have almost surely $\cid\cdot e_1<\xi\cdot e_1$.  The upshot is that 
$\P(\cid=\xi)=0$. 

Now fix $\xi\in(\ri\Uset)\smallsetminus\Diff$. By Lemma \ref{geo-sep} there exists a $z$ such that with positive probability geodesics $x^{0,\xi,\pm}_{0,\infty}$ separate at $z$.
This separation implies that $B^\xi_-(z,z+e_2) <B^\xi_-(z,z+e_1)$ and $B^\xi_+(z,z+e_1) < B^\xi_+(z,z+e_2)$, which says that
$\cid(T_z\what) = \xi$ and  thus $\xi$ is an atom of $\cid$.  We have proved the second  statement in part \eqref{th:two-geo:iii}.
 
The non-exposed points of $\ri\Uset$ consist of open linear segments  of $\gpp$ and the  endpoints of these segments  that lie in $\Diff$.  
 Consider a segment $[\zeta, \eta]\subset\ri\Uset$ on which $\gpp$ is linear.  Theorem \ref{th:cocycles}\eqref{th:cocycles:flat} says
$B_+^\zeta = B^\xi =  B_-^\eta$ for all  $\xi\in\;]\zeta, \eta [\,$.  Hence the inequalities in \eqref{ci}   go the same way throughout the segment and therefore  $\cid \in\;]\zeta, \eta [$ has zero probability.   Points in $\Diff$ were taken care of above. 
Since there are at most countably many   linear segments, 
the first  claim in part \eqref{th:two-geo:iii} follows.

Part \eqref{th:two-geo:iv}.  Assume first $\zeta\cdot e_1<\eta\cdot e_1$.   
 $\Uset_{\zeta+}\ne\Uset_{\eta-}$ and directedness   (Theorem \ref{th:B-geod:direction}) force   the   cocycle geodesics $x_{0,\infty}^{u,\eta,-}$ and $x_{0,\infty}^{u,\zeta,+}$   to eventually separate.  
This is clear if $\zetamax\ne\etamin$ because then $\Uset_{\zeta+}$ and $\Uset_{\eta-}$ are disjoint. 
If, on the other hand, $\zetamax=\etamin=\xi$, then $\nabla\gpp(\xi-)=\nabla\gpp(\zeta+)$ and $\nabla\gpp(\xi+)=\nabla\gpp(\eta-)$. By Theorem \ref{th:cocycles}\eqref{th:cocycles:flat} we have 
$x_{0,\infty}^{u,\zeta,+}=x_{0,\infty}^{u,\xi,-}$ and $x_{u,\infty}^{0,\eta,-}=x_{0,\infty}^{u,\xi,+}$. 
The separation claim then follows from Lemma \ref{geo-sep}. 

Now that we know the two geodesics separate at some random point $z$ we have almost surely 
$B^\zeta_+(z,z+e_2 )  <  B^\zeta_+(z,z+e_1 )$. By continuity \eqref{eq:cont} there is almost surely a rational $\zeta'\in\ri\Uset$ with $\zeta'\cdot e_1>\zeta\cdot e_1$  such that
$B^{\zeta'}_+(z,z+e_2 )  <  B^{\zeta'}_+(z,z+e_1 )$. Now we have $\zeta\cdot e_1 < \zeta'\cdot e_1\le\cid(T_z\what)\cdot e_1$.  A similar argument shows $\eta\cdot e_1>\cid(T_z\what)\cdot e_1$.   Thus $\cid(T_z\what)\in\;]\zeta,\eta[$.

Recall $B^*_\pm$ and $\t^\pm$  defined at and  below \eqref{B*} in terms of this $z=z(\what)$. 
Consider  two geodesics that start at $u$, follow  minimal  $\Bci_+$ and $\Bci_-$ gradients, and use  tie-breaking rules $\t^+$ and $\t^-$, respectively.
By monotonicity \eqref{eq:monotone} the two geodesics have to stay sandwiched between $x_{0,\infty}^{u,\zeta,+}$ and $x_{0,\infty}^{u,\eta,-}$
and therefore must pass through $z$.   Beyond $z$ these two geodesics are the ones discussed in the proof of part (i).   

In case (b) with $\xi\in(\ri\Uset)\smallsetminus\Diff$, Lemma \ref{geo-sep} gives the separation of $x^{u,\xi,\pm}_{0,\infty}$ at some random $z$.  Then $\cid(T_z\what)=\xi$ and the geodesics claimed in the theorem  are directly given by $x^{u,\xi,\pm}_{0,\infty}$.  \qed
\end{proof}


%

The next theorem identifies   the asymptotic direction of the competition interface $\varphi=(\varphi_k)_{0\le k<\infty}$ defined in Section \ref{sec:ci}.
 
  \begin{theorem}
Assume $\P\{\w_0\le r\}$ is continuous in $r$.
\begin{enumerate}[\ \ \rm(i)]
\item All limit points of the asymptotic velocity of the competition interface are in 
$\Uset_{\cid(\what)}$: for $\Phat$-almost every $\what$
\begin{align}\label{ci-lln1}
\cidmin(\what)\cdot e_1\le\varliminf_{n\to\infty} n^{-1}\varphi_n(\w)\cdot e_1\le\varlimsup_{n\to\infty} n^{-1}\varphi_n(\w)\cdot e_1\le\cidmax(\what)\cdot e_1.
\end{align}
\item If $\gpp$ is differentiable at the endpoints of its linear segments then $\cid$ is $\kS$-measurable and gives the asymptotic direction of the competition interface: $\Pbig$-almost surely
\begin{align}\label{ci-lln3}
\lim_{n\to\infty} n^{-1}\varphi_n(\w)=\cid(\what).
\end{align}
\end{enumerate}
\end{theorem}

\begin{proof}
By \eqref{ci}, if $\zeta\cdot e_1<\cid(\what)\cdot e_1<\eta\cdot e_1$, then 
$x_1^{0,\zeta,\pm}=e_2$
and $x_1^{0,\eta,\pm}=e_1$.
Since the path  $\varphi$ separates the geodesics that go through $e_1$ and $e_2$, it  has to stay  between 
$x_{0,\infty}^{0,\zeta,+}$ and $x_{0,\infty}^{0,\eta,-}$. By Theorem \ref{th:B-geod:direction} these geodesics are $\Uset_{\zeta+}$ and $\Uset_{\eta-}$ directed,
and we have
\[\underline\zeta\cdot e_1\le\varliminf_{n\to\infty} n^{-1}\varphi_n\cdot e_1\le\varlimsup_{n\to\infty} n^{-1}\varphi_n\cdot e_1\le\overline\eta\cdot e_1.\]
Claim \eqref{ci-lln1} follows by taking $\zeta$ and $\eta$ to $\cid$.

If $\gpp$ is differentiable at the endpoints of its linear segments, these endpoints are not exposed.  Since $\cid$ is  exposed by Theorem \ref{th:two-geo}\eqref{th:two-geo:iii},  we have  $\cidmin=\cidmax=\cid$ and 
  \eqref{ci-lln3} follows from  \eqref{ci-lln1}. Furthermore,  cocycles are $\kS$-measurable and hence so is $\cid$.  \qed
\end{proof}

\begin{proofof}{Proof of Theorem \ref{thm:ci-1}}  Limit \eqref{ci-lln2} is in \eqref{ci-lln3}.  The fact that the limit lies in $\ri\Uset$ is in the construction in the paragraph that contains \eqref{ci}, and the properties  of the limit    are in Theorem  \ref{th:two-geo}
parts \eqref{th:two-geo:iii} and \eqref{th:two-geo:iv}.   
\qed\end{proofof}

\begin{proofof}{of Theorem \ref{thm:ci-2}}
Under the assumption of differentiability at endpoints of linear segments,  either $\Uset_\xi$ equals $\{\xi\}$ or  $\Uset_\xi$ has no exposed points.  
Hence, by Theorem \ref{th:two-geo}\eqref{th:two-geo:iii},  almost surely $\Uset_{\cid}=\{\cid\}$  and  $\cid(T_x\what)\ne\xi$ implies that $\cid(T_x\what)\notin\Uset_\xi$.  
Consequently one of the cases in \eqref{eq:thm5.3(ii)} covers simultaneously  all $\Uset_\xi$-directed geodesics in environment $T_x\what$ and no separation at $x$ can happen for such geodesics. 
By Theorem \ref{thm1}\eqref{thm1:exist} every geodesic is $\Uset_\xi$-directed for some $\xi\in\Uset$. 
One direction in part  \eqref{thm:ci-2:(i)} is proved. The other direction comes from Theorem \ref{th:two-geo}\eqref{th:two-geo:i}.

 Part  \eqref{thm:ci-2:(ii)} comes  from part \eqref{thm:ci-2:(i)} and  Theorem \ref{th:two-geo}\eqref{th:two-geo:iv}.
 \qed\end{proofof}

As mentioned at the end of Section \ref{sec:ci}, if   $\P\{\w_0\le r\}$ is not continuous in $r$, we have competition interfaces $\varphi^{(l)}$ and $\varphi^{(r)}$ for the trees of leftmost and rightmost geodesics.  Their limiting directions $\cidr(\what),\,\cidl(\what)\in\ri\Uset$ are defined by 
 \begin{align}\label{ci12}
\begin{split}
&B^{\zeta}_{\pm}(\what, 0,e_1)> B^{\zeta}_{\pm}(\what,0,e_2)\quad\text{if }\zeta\cdot e_1<\cidr(\what)\cdot e_1,\\
&B^{\zeta}_{\pm}(\what,0,e_1)=B^{\zeta}_{\pm}(\what,0,e_2)\quad\text{if }\cidr(\what)\cdot e_1<\zeta\cdot e_1<\cidl(\what)\cdot e_1\\
\text{and}\qquad&B^{\zeta}_{\pm}(\what,0,e_1)<B^{\zeta}_{\pm}(\what,0,e_2)\quad\text{if }\zeta\cdot e_1>\cidl(\what)\cdot e_1.
\end{split}
\end{align}
With this definition    limit  \eqref{ci-lln1}  is valid also with superscripts $(l)$ and $(r)$.    Consequently    $n^{-1}\varphi^{(a)}_n(\w)\to \cid^{(a)}(\what)$ a.s.\ for $a\in\{l,r\}$ under the assumption that $\gpp$ is strictly concave.  

\section{Exactly solvable models}\label{sec:solv}
 
 


 
 We  derive here   \eqref{geom:theta}  for 
the distributions of $\cidr$ and $\cidl$  from  definition 
 \eqref{ci12}.  By  Sections~3.1 and 7.1 of \cite{Geo-Ras-Sep-15a-},  $B^{(a,1-a)}(0,e_1)$ and $B^{(a,1-a)}(0,e_2)$ are independent geometric random variables with means  
\begin{align*} 
\E[B^\xi(0, e_j)]&=\E(\w_0)+\sigma  \sqrt{\xi\cdot e_{3-j}/\xi\cdot e_j}, \quad j=1,2.
\end{align*}
 The calculation  for $\cidr$ goes   
\begin{align*} \P\{\cidr\cdot e_1>a\} &= \P\{ B^{(a,1-a)}(0,e_1)>B^{(a,1-a)}(0,e_2) \} \\
&= \frac{\sqrt{(\Ew-1)(1-a)}}{\sqrt{\Ew a}+ \sqrt{(\Ew-1)(1-a)}} 
\end{align*}
from which the first formula of \eqref{geom:theta} follows.   Similar computation for $\cidl$.

\appendix

\section{Auxiliary technical results}\label{app}

Cocycles satisfy a uniform ergodic theorem. The following is a special case of Theorem 9.3 of  \cite{Geo-etal-15-}. Note that a one-sided bound suffices for a hypothesis.   Recall  Definition \ref{def:cK} 
of  stationary $L^1(\P)$ cocycles. 
Let   $h(\B)\in\R^2$ denote the vector  that   satisfies 
	\[\E[\B(0,e_i)]=-h(\B)\cdot e_i \qquad\text{for  $i\in\{1,2\}$. }\]

\begin{theorem}\label{th:Atilla}
Assume $\P$ is ergodic under the group  $\{T_x\}_{x\in\Z^2}$. 
Let $B$ be a stationary $L^1(\P)$ cocycle.   Assume there exists  a function $V$ such that  for $\P$-a.e.\ $\w$  
\be \label{cL-cond}
\varlimsup_{\e\searrow0}\;\varlimsup_{n\to\infty} \;\max_{x: \abs{x}_1\le n}\;\frac1n \sum_{0\le k\le\e n} 
\abs{V(T_{x+ke_i}\w)}=0\qquad\text{for $i\in\{1,2\}$ }\ee  
and 
$\max_{i\in\{1,2\}} B(\w,0,e_i)\le V(\w)$.  
Then  
\[\lim_{n\to\infty}\;\max_{\substack{x=z_1+\dotsm+z_n\\z_{1,n}\in\{e_1, e_2\}^n}} \;\frac{\abs{B(\w,0,x)+h(B)\cdot x}}n=0 \qquad\text{for   $\P$-a.e.\ $\w$.}\]
\end{theorem}
If the process  $\{V(T_x\w):x\in\Z^2\}$ is  i.i.d.,  then a sufficient condition for \eqref{cL-cond} is  $\E(\abs{V}^p)<\infty$ for some $p>2$  \cite[Lemma A.4]{Ras-Sep-Yil-13}.\medskip

The following is a deterministic fact about gradients of passage times. This idea has been used profitably in planar percolation, and goes back at least to \cite{Alm-98, Alm-Wie-99}.    See Lemma~6.3 of \cite{Geo-Ras-Sep-15a-} for a proof.   

\begin{lemma}\label{lm:new:comp0}
Fix $\w\in\Omega$. Let $u,v\in\Z^2_+$ be such that $\abs{u}_1=\abs{v}_1\ge1$ and  $u\cdot e_1\le v\cdot e_1$. Then
	\begin{align}
	&\Gpp_{0,u}-\Gpp_{e_1,u}\ge\Gpp_{0,v}-\Gpp_{e_1,v}\quad\text{and}\label{auxlmeq1}\\
	&\Gpp_{0,u}-G_{e_2,u}\le\Gpp_{0,v}-\Gpp_{e_2,v}.\label{auxlmeq2}
	\end{align}
\end{lemma}

%
%
%

 \begin{acknowledgements}
N.\ Georgiou was partially supported by a Wylie postdoctoral fellowship at the University of Utah and the Strategic
Development Fund (SDF) at the University of Sussex.
F.\ Rassoul-Agha and N.\ Georgiou were partially supported by National Science Foundation grant DMS-0747758.
F.\ Rassoul-Agha was partially supported by National Science Foundation grant DMS-1407574 and by Simons Foundation grant 306576.
T.\ Sepp\"al\"ainen was partially supported by  National Science Foundation grants DMS-1306777 and DMS-1602486, by Simons Foundation grant 338287, and by  the Wisconsin Alumni Research Foundation.
\end{acknowledgements}



\bibliographystyle{spmpsci-nourlnodoi}      
\bibliography{firasbib2010}   

\begin{thebibliography}{10}
\providecommand{\url}[1]{{#1}}
\providecommand{\urlprefix}{URL }
\expandafter\ifx\csname urlstyle\endcsname\relax
  \providecommand{\doi}[1]{DOI~\discretionary{}{}{}#1}\else
  \providecommand{\doi}{DOI~\discretionary{}{}{}\begingroup
  \urlstyle{rm}\Url}\fi

\bibitem{Alm-98}
Alm, S.E.: A note on a problem by {W}elsh in first-passage percolation.
\newblock Combin. Probab. Comput. \textbf{7}(1), 11--15 (1998)

\bibitem{Alm-Wie-99}
Alm, S.E., Wierman, J.C.: Inequalities for means of restricted first-passage
  times in percolation theory.
\newblock Combin. Probab. Comput. \textbf{8}(4), 307--315 (1999).
\newblock Random graphs and combinatorial structures (Oberwolfach, 1997)

\bibitem{Auf-Dam-13}
Auffinger, A., Damron, M.: Differentiability at the edge of the percolation
  cone and related results in first-passage percolation.
\newblock Probab. Theory Related Fields \textbf{156}(1-2), 193--227 (2013)

\bibitem{Auf-Dam-Han-15}
Auffinger, A., Damron, M., Hanson, J.: Limiting geodesics for first-passage
  percolation on subsets of {$\Bbb{Z}^2$}.
\newblock Ann. Appl. Probab. \textbf{25}(1), 373--405 (2015)

\bibitem{Bak-07}
Bakhtin, Y.: Burgers equation with random boundary conditions.
\newblock Proc. Amer. Math. Soc. \textbf{135}(7), 2257--2262 (electronic)
  (2007)

\bibitem{Bak-13}
Bakhtin, Y.: The {B}urgers equation with {P}oisson random forcing.
\newblock Ann. Probab. \textbf{41}(4), 2961--2989 (2013)

\bibitem{Bak-15-}
Bakhtin, Y.: Inviscid burgers equation with random kick forcing in noncompact
  setting  (2014).
\newblock Preprint ({\tt arXiv:1406.5660})

\bibitem{Bak-Cat-Kha-14}
Bakhtin, Y., Cator, E., Khanin, K.: Space-time stationary solutions for the
  {B}urgers equation.
\newblock J. Amer. Math. Soc. \textbf{27}(1), 193--238 (2014)

\bibitem{Bak-Kha-10}
Bakhtin, Y., Khanin, K.: Localization and {P}erron-{F}robenius theory for
  directed polymers.
\newblock Mosc. Math. J. \textbf{10}(4), 667--686, 838 (2010)

\bibitem{Cat-Pim-11}
Cator, E., Pimentel, L.P.R.: A shape theorem and semi-infinite geodesics for
  the {H}ammersley model with random weights.
\newblock ALEA Lat. Am. J. Probab. Math. Stat. \textbf{8}, 163--175 (2011)

\bibitem{Cat-Pim-12}
Cator, E., Pimentel, L.P.R.: Busemann functions and equilibrium measures in
  last passage percolation models.
\newblock Probab. Theory Related Fields \textbf{154}(1-2), 89--125 (2012)

\bibitem{Cat-Pim-13}
Cator, E., Pimentel, L.P.R.: Busemann functions and the speed of a second class
  particle in the rarefaction fan.
\newblock Ann. Probab. \textbf{41}(4), 2401--2425 (2013)

\bibitem{Coh-Elk-Pro-96}
Cohn, H., Elkies, N., Propp, J.: Local statistics for random domino tilings of
  the {A}ztec diamond.
\newblock Duke Math. J. \textbf{85}(1), 117--166 (1996)

\bibitem{Cou-11}
Coupier, D.: Multiple geodesics with the same direction.
\newblock Electron. Commun. Probab. \textbf{16}, 517--527 (2011)

\bibitem{Dam-Han-14}
Damron, M., Hanson, J.: Busemann functions and infinite geodesics in
  two-dimensional first-passage percolation.
\newblock Comm. Math. Phys. \textbf{325}(3), 917--963 (2014)

\bibitem{Dur-Lig-81}
Durrett, R., Liggett, T.M.: The shape of the limit set in {R}ichardson's growth
  model.
\newblock Ann. Probab. \textbf{9}(2), 186--193 (1981)

\bibitem{E-etal-00}
E, W., Khanin, K., Mazel, A., Sinai, Y.: Invariant measures for {B}urgers
  equation with stochastic forcing.
\newblock Ann. of Math. (2) \textbf{151}(3), 877--960 (2000)

\bibitem{Fer-Kip-95}
Ferrari, P.A., Kipnis, C.: Second class particles in the rarefaction fan.
\newblock Ann. Inst. H. Poincar{\'e} Probab. Statist. \textbf{31}(1), 143--154
  (1995)

\bibitem{Fer-Mar-Pim-09}
Ferrari, P.A., Martin, J.B., Pimentel, L.P.R.: A phase transition for
  competition interfaces.
\newblock Ann. Appl. Probab. \textbf{19}(1), 281--317 (2009)

\bibitem{Fer-Pim-05}
Ferrari, P.A., Pimentel, L.P.R.: Competition interfaces and second class
  particles.
\newblock Ann. Probab. \textbf{33}(4), 1235--1254 (2005)

\bibitem{Gar-Mar-05}
Garet, O., Marchand, R.: Coexistence in two-type first-passage percolation
  models.
\newblock Ann. Appl. Probab. \textbf{15}(1A), 298--330 (2005)

\bibitem{Geo-Ras-Sep-15-}
Georgiou, N., Rassoul-Agha, F., Sepp\"al\"ainen, T.: Geodesics and the
  competition interface for the corner growth model  (2015).
\newblock Preprint ({\tt arXiv:1510.00860v1})

\bibitem{Geo-Ras-Sep-15a-}
Georgiou, N., Rassoul-Agha, F., Sepp\"al\"ainen, T.: Stationary cocycles and
  {B}usemann functions for the corner growth model.
\newblock Probab. Theory Relat. Fields  (2016).
\newblock To appear ({\tt arXiv:1510.00859})

\bibitem{Geo-etal-15-}
Georgiou, N., Rassoul-Agha, F., Sepp\"al\"ainen, T., Y{\i}lmaz, A.: Ratios of
  partition functions for the log-gamma polymer.
\newblock Ann. Probab. \textbf{43}(5), 2282--2331 (2015)

\bibitem{Hoa-Kha-03}
Hoang, V.H., Khanin, K.: Random {B}urgers equation and {L}agrangian systems in
  non-compact domains.
\newblock Nonlinearity \textbf{16}(3), 819--842 (2003)

\bibitem{Hof-05}
Hoffman, C.: Coexistence for {R}ichardson type competing spatial growth models.
\newblock Ann. Appl. Probab. \textbf{15}(1B), 739--747 (2005)

\bibitem{Hof-08}
Hoffman, C.: Geodesics in first passage percolation.
\newblock Ann. Appl. Probab. \textbf{18}(5), 1944--1969 (2008)

\bibitem{How-New-01}
Howard, C.D., Newman, C.M.: Geodesics and spanning trees for {E}uclidean
  first-passage percolation.
\newblock Ann. Probab. \textbf{29}(2), 577--623 (2001)

\bibitem{Itu-Kha-03}
Iturriaga, R., Khanin, K.: Burgers turbulence and random {L}agrangian systems.
\newblock Comm. Math. Phys. \textbf{232}(3), 377--428 (2003)

\bibitem{Joc-Pro-Sho-98}
Jockusch, W., Propp, J., Shor, P.: Random domino tilings and the arctic circle
  theorem  (1998).
\newblock {\tt arXiv:math/9801068}

\bibitem{Lic-New-96}
Licea, C., Newman, C.M.: Geodesics in two-dimensional first-passage
  percolation.
\newblock Ann. Probab. \textbf{24}(1), 399--410 (1996)

\bibitem{Mai-Pra-03}
Mairesse, J., Prabhakar, B.: The existence of fixed points for the
  {$\cdot/GI/1$} queue.
\newblock Ann. Probab. \textbf{31}(4), 2216--2236 (2003)

\bibitem{Mar-02}
Marchand, R.: Strict inequalities for the time constant in first passage
  percolation.
\newblock Ann. Appl. Probab. \textbf{12}(3), 1001--1038 (2002)

\bibitem{Mar-04}
Martin, J.B.: Limiting shape for directed percolation models.
\newblock Ann. Probab. \textbf{32}(4), 2908--2937 (2004)

\bibitem{Mou-Gui-05}
Mountford, T., Guiol, H.: The motion of a second class particle for the {TASEP}
  starting from a decreasing shock profile.
\newblock Ann. Appl. Probab. \textbf{15}(2), 1227--1259 (2005)

\bibitem{New-95}
Newman, C.M.: A surface view of first-passage percolation.
\newblock In: Proceedings of the {I}nternational {C}ongress of
  {M}athematicians, {V}ol.\ 1, 2 ({Z}{\"u}rich, 1994), pp. 1017--1023.
  Birkh{\"a}user, Basel (1995)

\bibitem{Pim-07}
Pimentel, L.P.R.: Multitype shape theorems for first passage percolation
  models.
\newblock Adv. in Appl. Probab. \textbf{39}(1), 53--76 (2007)

\bibitem{Pim-13-}
Pimentel, L.P.R.: Duality between coalescence times and exit points in
  last-passage percolation models.
\newblock Ann. Probab.  (2015).
\newblock To appear ({\tt arXiv:1307.7769})

\bibitem{Pra-03}
Prabhakar, B.: The attractiveness of the fixed points of a {$\cdot/GI/1$}
  queue.
\newblock Ann. Probab. \textbf{31}(4), 2237--2269 (2003)

\bibitem{Ras-Sep-Yil-13}
Rassoul-Agha, F., Sepp{\"a}l{\"a}inen, T., Y{\i}lmaz, A.: Quenched free energy
  and large deviations for random walks in random potentials.
\newblock Comm. Pure Appl. Math. \textbf{66}(2), 202--244 (2013)

\bibitem{Ros-81}
Rost, H.: Nonequilibrium behaviour of a many particle process: density profile
  and local equilibria.
\newblock Z. Wahrsch. Verw. Gebiete \textbf{58}(1), 41--53 (1981)

\bibitem{Sep-98-mprf-1}
Sepp{{\"a}}l{{\"a}}inen, T.: Hydrodynamic scaling, convex duality and
  asymptotic shapes of growth models.
\newblock Markov Process. Related Fields \textbf{4}(1), 1--26 (1998)

\bibitem{Sep-99-aop}
Sepp{{\"a}}l{{\"a}}inen, T.: Existence of hydrodynamics for the totally
  asymmetric simple {$K$}-exclusion process.
\newblock Ann. Probab. \textbf{27}(1), 361--415 (1999)

\bibitem{Weh-97}
Wehr, J.: On the number of infinite geodesics and ground states in disordered
  systems.
\newblock Journal of Statistical Physics \textbf{87}(1-2), 439--447 (1997)

\bibitem{Weh-Woo-98}
Wehr, J., Woo, J.: Absence of geodesics in first-passage percolation on a
  half-plane.
\newblock Ann. Probab. \textbf{26}(1), 358--367 (1998)

\bibitem{Wut-02}
W{{\"u}}thrich, M.V.: Asymptotic behaviour of semi-infinite geodesics for
  maximal increasing subsequences in the plane.
\newblock In: In and out of equilibrium ({M}ambucaba, 2000), \emph{Progr.
  Probab.}, vol.~51, pp. 205--226. Birkh{\"a}user Boston, Boston, MA (2002)

\end{thebibliography}


\end{document}